\documentclass[12pt]{amsart}
\usepackage[dvips]{color}
\usepackage{amsmath}
\usepackage{amsxtra}
\usepackage{amscd}
\usepackage{amsthm}
\usepackage{amsfonts}
\usepackage{amssymb}
\usepackage{eucal}
\usepackage{epsfig}
\usepackage{graphics}
\def\({\left(}
\def\){\right)}

\newcommand{\gl}{\mathfrak{gl}}

\newcommand{\ga}{\gamma}
\newcommand\Ref{\eqref}

\newcommand{\bra}[1]{\langle #1 |}        
\newcommand{\ket}[1]{{| #1 \rangle}}      

\newcommand{\nn}{\nonumber}
\newcommand{\bea}{\begin{eqnarray}}
\newcommand{\ena}{\end{eqnarray}}
\def\bel{\begin{eqnarray}}
\def\enl{\end{eqnarray}}
\newcommand{\be}{\begin{eqnarray*}}
\newcommand{\en}{\end{eqnarray*}}

\newcommand{\C}{{\mathbb C}}
\newcommand{\Z}{{\mathbb Z}}
\newcommand{\Q}{{\mathbb Q}}

\newcommand{\mc}{\mathcal}

\newenvironment{tenumerate}{
  \begin{enumerate}
  
  }{\end{enumerate}}
\newcommand{\bi}{\begin{tenumerate}}
\newcommand{\ei}{\end{tenumerate}}
\newcommand{\isoto}[1][]%
{{\mathop{\buildrel{\sim}\over\longrightarrow}\limits_{#1}}}

\def\[{\left[}
\def\]{\right]}
\newcommand{\la}{\lambda}
\newcommand{\La}{\Lambda}

\newcommand{\al}{\alpha}

\newcommand{\bs}{\boldsymbol}

\numberwithin{equation}{section}
\newtheorem{thm}{Theorem}[section]

\newtheorem{lem}[thm]{Lemma}
\newtheorem{rem}[thm]{Remark}
\newtheorem{cor}[thm]{Corollary}

\newtheorem{conj}[thm]{Conjecture}

\newcommand{\E}{{\mathcal E}}

\newcommand{\on}{\operatorname}

\newcommand{\F}{\mathcal F}

\def\bi{\mathbf{i}}

\def\mdf#1{#1}
\definecolor{4.2}{rgb}{0,0,0}
\definecolor{jimbo}{rgb}{0,0,0}

\begin{document}
\begin{title}[Branching rules for quantum toroidal  $\mathfrak{gl}_n$.]
{Branching rules for quantum toroidal  $\mathfrak{gl}_n$}
\end{title}
\author{B. Feigin, M. Jimbo, T. Miwa, and E. Mukhin}
\address{BF: Landau Institute for Theoretical Physics,
Russia, Chernogolovka, 142432, prosp. Akademika Semenova, 1a,   \newline
Higher School of Economics, Russia, Moscow, 101000,  Myasnitskaya ul., 20 and
\newline
Independent University of Moscow, Russia, Moscow, 119002,
Bol'shoi Vlas'evski per., 11}
\email{bfeigin@gmail.com}
\address{MJ: Department of Mathematics,
Rikkyo University, Toshima-ku, Tokyo 171-8501, Japan}
\email{jimbomm@rikkyo.ac.jp}
\address{TM: Institute for Liberal Arts and Sciences,
Kyoto University, Kyoto 606-8316,
Japan}\email{tmiwa@kje.biglobe.ne.jp}
\address{EM: Department of Mathematics,
Indiana University-Purdue University-Indianapolis,
402 N.Blackford St., LD 270,
Indianapolis, IN 46202, USA}\email{mukhin@math.iupui.edu}

\begin{abstract} We construct 
\textcolor{jimbo}{an analog} 
of the subalgebra 
$U\mathfrak{gl}(n)\otimes U\mathfrak{gl}(m)\subset U\mathfrak{gl}(m+n)$ 
in the setting of quantum toroidal algebras and 
study the restrictions of various representations to this subalgebra.
\end{abstract}

\maketitle
\section{Introduction}
\subsection{Motivation: the AGT conjecture}
The quantum toroidal algebra, \cite{GKV}, associated with a semi-simple Lie algebra $\mathfrak g$ 
is the quantum version of the universal enveloping algebra of the Lie algebra of currents
$\C^*\times\C^*\rightarrow\mathfrak g$. 

In this paper we consider only the case $\mathfrak g=\mathfrak{gl}_n$, $n\geq 1$.
The corresponding toroidal algebra
$\mathcal E_n=\mathcal E_n(q_1,q_2,q_3)$, see Section \ref{generators}, depends
on three deformation parameters $q_1,q_2,q_3$ such that $q_1q_2q_3=1$.
The algebra $\mathcal E_n(q_1,q_2,q_3)$ has two central elements
which we denote by $q^c$ and $\kappa$. In all representations
appearing in this paper, one of the central elements, $q^c$, always acts by $1$. 
In the limit $q_2\to 1$, the algebra $\mc E_n$ becomes the universal central extension of the universal enveloping algebra of the Lie algebra
$\mathbb{M}_n\otimes\C[Z^{\pm1},D^{\pm1}]$, see Section \ref{classical}.  Here $\mathbb{M}_n$ is the algebra of $n\times n$ matrices, and
$\C[Z^{\pm1},D^{\pm1}]$ is the algebra of functions on the one-dimensional quantum torus:
$ZD=q_1^nDZ$. The Lie algebra structure is given by the standard formula $[a,b]=ab-ba$.

\medskip

The algebra $\mathcal E_n$ has another important so-called conformal limit. This limit is more subtle and it is obtained by setting $q_1=\varepsilon^{\sigma_1}$, $q_2=\varepsilon^{\sigma_2}$,
$q_3=\varepsilon^{\sigma_3}$ with $\sigma_1+\sigma_2+\sigma_3=0$, $\kappa=\varepsilon^k$, and sending $\varepsilon \to 1$. This limit is called conformal, because the limiting algebra has a vertex operator algebra (conformal algebra) structure. The limiting algebra depends on $\sigma_1/\sigma_2$ and $k$. 
Note that the algebra obtained via the conformal limit for special values of parameters is smaller than $\mc E_n$.

The conformal limit is important for the study of the AGT conjecture. The AGT conjecture, \cite{AGT}, 
claims that when parameters of the $4$-dimensional topological super Yang-Mills
field theory go to an appropriate limit, the theory becomes deeply connected to a 
$2$-dimensional conformal field theory. At the same time the algebra $\mathcal E_n(q_1,q_2,q_3)$ acts by correspondences in the space of the $K$-theory of the moduli spaces of instantons related to the $4$-dimensional
topological super Yang-Mills field theory and the conformal limit of $\mathcal E_n(q_1,q_2,q_3)$ describes the relevant conformal field theory.

\subsection{Motivation: the coset constructions}
Consider a pair of affine Lie algebras: $\widehat{\mathfrak{gl}}_N$ and its subalgebra
$\widehat{\mathfrak{gl}}_{N-n}\subset\widehat{\mathfrak{gl}}_N$ both with level $k$.
The well-known coset construction  of conformal field theory
gives a new vertex operator algebra for this pair, which we denote
$\mathcal C_k(\widehat{\mathfrak{gl}}_N,\widehat{\mathfrak{gl}}_{N-n})$.
The coset algebras naturally appear in the problem of decompositions of representations.
Consider a restriction of an integrable representation $\pi$ of $\widehat{\mathfrak{gl}}_N$ with level $k$
to the subalgebra $\widehat{\mathfrak{gl}}_{N-n}$. Then 
we have the decomposition $\pi=\oplus_\alpha\ W_\alpha\otimes \mathcal R_\alpha$,
where $\mathcal R_\alpha$ are irreducible representations of $\widehat{\mathfrak{gl}}_{N-n}$, and
spaces of multiplicities $W_\alpha$ are irreducible 
\textcolor{jimbo}{representations} 
of the algebra
$\mathcal C_k(\widehat{\mathfrak{gl}}_N,\widehat{\mathfrak{gl}}_{N-n})$.

The problem of decomposition of $U\mathfrak{gl}_N$ module after restriction to $U\mathfrak{gl}_{N-n}$ is closely related to the problem of finding the commutant of the subalgebra $U\mathfrak{gl}_{N-n}$ in $U\mathfrak{gl}_N$.
This commutant can be described explicitly, and it is closely related to the Yangian of $\mathfrak{gl}_n$, see  \cite{O1}, \cite{O2}. Namely, the commutant is a factor of the Yangian and the Yangian can be viewed as the analytic continuation of the commutant with respect to the variable $N$. To get generators and relations of the coset algebra $\mathcal C_k(\widehat{\mathfrak{gl}}_N,\widehat{\mathfrak{gl}}_{N-n})$ one has to study the commutant in the affine setting. 

Clearly, $\mathcal C_k(\widehat{\mathfrak{gl}}_N,\widehat{\mathfrak{gl}}_{N-n})$ contains the subalgebra $\widehat{\mathfrak{gl}}_n$ with level $k$ generated by $E_{ij}(z)=\sum_{s\in\Z}(E_{ij}\otimes t^s)z^{-s}$, where $i,j=N-n+1,\dots,N$. It also contains the quadratic currents $E^{(2)}_{ij}(z)=\sum_{\alpha=1}^{N-n} :E_{i\alpha }(z)E_{\alpha j}(z):$
with $i,j=N-n+1,\dots,N$. In fact, the algebra $\mathcal C_k(\widehat{\mathfrak{gl}}_N,\widehat{\mathfrak{gl}}_{N-n})$
is generated by $E_{ij}(z)$ and $E^{(2)}_{ij}(z)$. But in the operator product of the currents
$E_{ij}(z)$ with $E_{ij}(w)$ one can find cubic currents $E^{(3)}_{ij}(z)$,
then quartic currents $E^{(4)}_{i,j}(z)$ and so on. The coset algebra
$\mathcal C_k(\widehat{\mathfrak{gl}}_N,\widehat{\mathfrak{gl}}_{N-n})$ is expected to be a factor of a quantization
of the universal enveloping algebra of the double current Lie algebra $\mathfrak{gl}_n\otimes\C[z_1^{\pm1},z_2]$, and then the
currents $E^{(m)}_{ij}(z)$ should correspond to the currents $E_{ij}(z_1)z_2^m$.

\medskip 

The case of $\mathcal C_k(\widehat{\mathfrak{gl}}_N,\widehat{\mathfrak{gl}}_{N-1})$ 
is the best studied and is known as the $W$-algebra associated to $\gl_k$. There exists a number of alternative constructions which produce the $W$-algebra, though in almost all cases, a rigorous proof of the identification is missing. 

For example, consider the algebra obtained by the quantum Drinfeld-Sokolov reduction of $\widehat{\mathfrak{gl}}_M$ with level $s$, followed by the analytic continuation with respect to $M$ \cite{FF}.
We follow the standard notation and denote the result by $W_{M,\frac1{s+M}}$. Then with this notation, we have
$\mathcal C_k(\widehat{\mathfrak{gl}}_N,\widehat{\mathfrak{gl}}_{N-1})\simeq
W_{k,\frac{N+k+1}{N+k+2}}$. This statement is non-trivial, the direct check is tedious and has not been done yet. 

There exists a dual coset construction of the algebra $\mathcal C_k(\widehat{\mathfrak{gl}}_N,\widehat{\mathfrak{gl}}_{N-1})$, where one takes
$\widehat{\mathfrak{gl}}_k$ of level $1$ times $\widehat{\mathfrak{gl}}_k$
of arbitrary level and considers the coset with respect to the diagonal embedding
of $\widehat{\mathfrak{gl}}_k$. 

There is an additional puzzling observation that the algebra 
$\mathcal C_k(\widehat{\mathfrak{gl}}_N,\widehat{\mathfrak{gl}}_{N-1})$
is isomorphic \textcolor{jimbo}{to}
the $W$-algebra constructed by the Drinfeld-Sokolov reduction from
\textcolor{jimbo}{the}
Lie superalgebra $\widehat{\mathfrak{gl}}(N|N-1)$.

\medskip

The algebra $\mathcal C_k(\widehat{\mathfrak{gl}}_N,\widehat{\mathfrak{gl}}_{N-n})$
depends on two parameters $N,k$, where $k$ is the level of $\widehat{\mathfrak{gl}}_N$.
The parameter $k$ is a complex number, while $N$ is natural. 
However, the structure constant in $\mathcal C_k(\widehat{\mathfrak{gl}}_N,\widehat{\mathfrak{gl}}_{N-n})$
depends on $N$ algebraically. Therefore, we can make the analytic
continuation with respect to $N$, then $N$ becomes an arbitrary complex number. 
The quantum toroidal algebra \textcolor{jimbo}{ $\mathcal E_n(q_1,q_2,q_3)$ }
is a quantization of the resulting algebra. 
Moreover, the conformal limit
of algebra $\mathcal E_n(q_1,q_2,q_3)$ coincides with $\mathcal C_k(\widehat{\mathfrak{gl}}_N,\widehat{\mathfrak{gl}}_{N-n})$.

In particular, the algebra corresponding to the $W$-algebra, 
\textcolor{jimbo}{$\mathcal E_1(q_1,q_2,q_3)$}, 
is the quantum toroidal algebra which has been most extensively studied. It is known as elliptic Hall algebra \cite{BS}, \cite{S}, \cite{SV2},
$(q,\gamma)$ analog of $\mathcal W_{1+\infty}$, \cite{M07},
an elliptic deformation of the $W$ algebra of type $\mathfrak{gl}$, Ding-Iohara algebra, \cite{FHHSY}, spherical Cherednik DAHA \cite{SV1}, quantum continuous $\mathfrak{gl}_\infty$, \cite{FFJMM1}, \cite{FFJMM2}.

\medskip

The coset construction has \textcolor{jimbo}{a} 
quantum group version.
Consider the quantum affine algebra $U_q(\widehat{\mathfrak{gl}}_N)$
with the subalgebra $U_q(\widehat{\mathfrak{gl}}_{N-1})$.
Then the problem is to find the commutant of $U_q(\widehat{\mathfrak{gl}}_{N-1})$
in $U_q(\widehat{\mathfrak{gl}}_N)$. This is a non-trivial question which we suggest to 
solve using the quantum toroidal algebras. 

Namely, one expects that there is an evaluation map
$\mathcal E_N(q_1,q_2,q_3)\rightarrow U_q(\widehat{\mathfrak{gl}}_N)$
where $q^2=q_2$ and the level of $\widehat{\mathfrak{gl}}_N$ depends on $q_1$ and $\kappa$.
Then on the level of quantum toroidal algebras, we find a homomorphism of algebras
$\varphi:\mathcal E_1\otimes\mathcal E_{N-1}\rightarrow\tilde{\mathcal E}_N$
where $\tilde{\mathcal E}_N$ is a suitable completion of $\mathcal E_N$.
In the Lie algebra limit $q_2 \to 1$ the map $\varphi$ becomes
very simple: it is just the map coming from the embedding 
\begin{align*}
\mathbb{M}_1\otimes\C[Z^{\pm 1},D^{\pm1}]\oplus
\mathbb{M}_{N-1}\otimes\C[Z^{\pm1},D^{\pm1}]
\rightarrow
\mathbb{M}_N\otimes\C[Z^{\pm1},D^{\pm1}].
\end{align*}
Note that on the other hand the conformal limit of $\varphi$ is rather non-trivial.

Combining $\varphi$ with the evaluation map, we obtain
\begin{align*}
\mathcal E_1\otimes\mathcal E_{N-1}\rightarrow U_q(\widehat{\mathfrak{gl}}_N).
\end{align*}
The image of the subalgebra $1\otimes\mathcal E_{N-1}$ is 
$U_q(\widehat{\mathfrak{gl}}_{N-1})$
and $\mathcal E_1\otimes 1$ is mapped to the commutant of
$U_q(\widehat{\mathfrak{gl}}_{N-1})$ in $U_q(\widehat{\mathfrak{gl}}_N)$.
Actually, we believe that the map of algebra $\mathcal E_1 \otimes1$ to the commutant is surjective, 
but we do not discuss this fact in the present paper.
Instead we concentrate on a family of irreducible representations of the algebra
$\mathcal E_N$ and study the restriction on the product
$\mathcal E_1\otimes\mathcal E_{N-1}$. In all  cases we consider, the multiplicities of irreducible representations of $\mathcal E_1\otimes\mathcal E_{N-1}$ appearing in irreducible representations of 
$\mathcal E_N$ are one.

\subsection{Motivation: geometry} The simplest integrable representation of $\mc E_n$ is called the  
Fock module, \cite{VV2},  \cite{STU}, \cite{FJMM1}, \cite{FJMM2}, \cite{S}. The Fock module appears in geometry in the following way. Consider the Hilbert scheme
$H_d$ of ideals of codimension $d$ in $\C[z_1,z_2]$. The plane $\C^2$ is equipped with an action of the torus  $\C^*\times\C^*$ via
$\alpha\times\beta:\ (z_1,z_2)\mapsto(\alpha z_1,\beta z_2)$ and of the cyclic group $\Z_p$ of order $p$
via $\zeta(z_1,z_2)=(\zeta z_1,\zeta^{-1}z_2)$, where $\zeta\in\C^*$ is a root of unity of order $p$. 
These actions induce the corresponding actions in $H_d$.

Let $H^{(p)}_d$ be the manifold of
the fixed points of $\Z_p$ in $H_d$. The manifold $H^{(p)}_d$ is smooth but not connected.
It is known that the quantum toroidal algebra $\mathcal E_p(q_1,q_2,q_3)$ acts
in the equivariant $K$-theory space $\mathcal F=\oplus_{d=0}^\infty K(H^{(p)}_d)$, where $q_1,q_2$ are the
equivariant parameters, by correspondences, see \cite{N}, \cite{FT}.
This representation of $\mc E_p$ is isomorphic to the Fock module.
Moreover, geometrically one observes the following remarkable phenomenon.

A basis in $\mathcal F$ is given by fixed points of $\C^*\times\C^*$ action.
This basis consists of eigenvectors of the Cartan subalgebra of $\mathcal E_p$.
If  $J\in H^{(p)}_d$ then $J\subset\C[z_1,z_2]$ is a homogeneous ideal
such that the quotient $\C[z_1,z_2]/J$ is a $d$-dimensional representation of $\Z_p$.
Irreducible representations of $\Z_p$ are all one dimensional, denote them
$\nu_0,\nu_1,\dots,\nu_{p-1}$. We call $J\in H^{(p)}_d$ of type $(a_0,\dots,a_{p-1})$ if 
$\C[z_1,z_2]/J = \oplus_{i=0}^{p-1} a_i\nu_i$.  Note that $a_0+\dots+a_{p-1}=d$.
Denote $H_{a_0,\dots,a_{p-1}}\subset H^{(p)}_d$ the set of ideals of type $(a_0,\dots, a_{p-1})$.

Then $H_{a_0,\dots,a_{p-1}}$ are exactly the connected components of $H^{(p)}_d$, and we have the geometric description of the weight decomposition of the Fock module: $\mathcal F=\oplus K(H_{a_1,\ldots,a_{p-1}})$.

The algebra $\mathcal E_p$ has a large group of automorphisms which is a toroidal version
of Lusztig braid group, see \cite{M99}. In particular, this group contains the root lattice of $\mathfrak{sl}_p$, which consists of the extensions of the affine translations to $\mc E_p$. This  
lattice is isomorphic to $\Z^{p-1}$ and it also acts in the Fock module $\mathcal F$.
Geometric description of the action of the braid group is non-trivial, but one can observe the following stabilization of manifolds.

The group $\Z^{p-1}$ acts
in the set of weights $\{(a_0,\ldots,a_{p-1})\}$. Fix some $A=(a_0,\ldots,a_{p-1})$ and let $T\in\Z^{p-1}$ be a generic element. Consider the sequence of manifolds $\mathcal M_s=H^{(p)}_{T^s A}$, $s=0,1,2,\dots$.
According to \cite{N}, for $s$ large enough the manifolds $\mathcal M_s$ are all isomorphic and
have a simple geometric description which can be described as follows.
Consider the quotient $(\C\times\C)/\Z_p$. It has the Kleinian singularity at the origin. Resolve this singularity and call the result $X_p$.
Then $X_p$ is a $2$-dimensional smooth manifold with a natural action of the torus $\C^*\times\C^*$.

For $s$ large enough the manifolds $\mathcal M_s$ is isomorphic
to a connected component of the Hilbert scheme of torsion free sheaves on $X_p$. The choice of the connected component corresponds to the choice of $(a_0,\ldots,a_{p-1})$.

On the other hand, the manifold $X_p$ has $p$ fixed points with respect to $\C^*\times\C^*$.
Therefore, on the Hilbert scheme of $X_p$ we have $p$ commuting actions 
of $\mc E_1$. The $i$-th action is given by correspondences with support in
the $i$-th point. Thus, in this limit of the Fock module, we observe an action of
$\mc E_1^{\otimes p}$.
One of the goals of this paper is to give a representation-theoretic explanation of this phenomenon.

Namely, the Cartan subalgebra of $\mc E_p$ is a commutative algebra generated by $\{K^\pm_i(z)\}$.
The fixed points are eigenvectors with respect to the operators $K^\pm_i(z)$.
Consider the operators $T^sK^\pm_i(z)$, $s=0,1,2,\dots$, acting in the Fock module. 
For $v\in\mc F$, we have  $T^sK^\pm_i(z) \cdot v=T^s \circ K^\pm_i(z)\circ T^{-s} v$.
For each $v\in\mc F$, for large enough $s$, the vector $T^sK^\pm_i(z) \cdot v$ does not depend on $s$.
The joint spectrum of the Cartan subalgebra is simple, so in the limit $s\to \infty$, we obtain 
a basis of the Fock module.

In addition, we construct an embedding 
\textcolor{jimbo}{$\mc E_1^{\otimes p}\to \widetilde{\mathcal{E}}_p$. }
Then the action of this subalgebra in the above basis recovers the geometric action.

\subsection{The plan of the paper and the main results}
Here is the outline of the paper.

We denote $\mc E_n$ the quantum toroidal algebra of type $\gl_n$.

Section \ref{toroidal-algebra} collects notation and basic facts about $\mc E_n$.
We discuss \textcolor{jimbo}{the}
defining relations in Section \ref{generators}, automorphisms in Section \ref{aut sec}, representation theory in Section \ref{rep sec}. In the literature, the cases $n\geq 3$ and $n=1$ usually appear separately, while $n=2$ is often omitted. We manage to write all formulas in a uniform way.

Section \ref{subalgebras} contains the construction and \textcolor{jimbo}{the}
properties of $\mc E_m$ inside a suitable completion of $\mc E_n$, $m<n$. 
The main construction is  described in Section \ref{fused current}, 
it defines fused currents via a quantum version of \textcolor{jimbo}{the}
operator product 
expansion. Then we prove our first main results, 
Theorem \ref{sub alg}, see Section \ref{relations} and 
Theorem \ref{commute}, see Section \ref{commute sec}. Theorem \ref{sub alg} establishes that the fused currents do satisfy the relations of the quantum toroidal $\gl_m$, and Theorem \ref{commute} proves that \textcolor{jimbo}{the}
upper left corner and \textcolor{jimbo}{the}
bottom right corner subalgebras $\mc E_n$ and $\mc E_m$  commute within $\mc E_{m+n}$. Our main method is the study of correlation functions, we develop the techniques in Section \ref{cor func}. In Section \ref{classical} we describe the Lie algebra limit of $\mc E_n$ and the meaning of our construction in this limit.

Section \ref{branching sec} is devoted to the study of the modules over $\mc E_{m+n}$ after restriction to $\mc E_m\otimes\mc E_n$. We write the formulas mainly in the case of $n=1$. We give all details in the case of the Fock module and $n=1$ to explain the approach and the logic of the proofs, see Section \ref{fock branch sec}. Then we proceed to tensor products of Fock modules and their irreducible submodules. 
The main results are Theorem \ref{Fock n thm}, Theorem \ref{k Fock n thm} and Theorem \ref{N thm}. These theorems explicitly describe decompositions of various modules. We conclude with a conjectural formula for the decomposition of the so called Macmahon module.

\section{Quantum toroidal algebras}\label{toroidal-algebra}
In this section we introduce our notation concerning 
the quantum toroidal algebra of type $\gl_n$.
We also recall its basic features relevant to the present text. 

\subsection{Generators and relations}\label{generators}
Let $n$ be a natural number. We shall write $a\equiv b$ for $a\equiv b\bmod n$. 
Let $(a_{i,j})_{i,j=0}^{n-1}$ be the Cartan matrix  of type $A^{(1)}_{n-1}$, 
and let $(m_{i,j})_{i,j=0}^{n-1}$ be a skew-symmetric matrix defined by 
$m_{i+1,i}=1$  and $m_{i,j}=0$ if $i\not\equiv j\pm1$, 
where the suffix is to be read modulo $n$. 

Fix non-zero complex numbers $d,q$. 
Throughout the text we shall use the parameters
\begin{align*}
& q_1=d q^{-1},\ q_2=q^2,\ q_3=d^{-1}q^{-1}\,,
\end{align*}
so that $q_1q_2q_3=1$. We assume further that for $n_1,n_2,n_3\in\Z$
\begin{align*}
\text{$q_1^{n_1}q_2^{n_2}q_3^{n_3}=1$ holds only if $n_1=n_2=n_3$}.
\end{align*}
In particular,  none of the $q_i$ is a root of unity.

The {\it quantum toroidal algebra} of type $\gl_n$,  which we denote 
$\mathcal{E}_n$, 
is an associative 
unital $\C$-algebra defined by generators and relations to be given below. 

The algebra $\mathcal{E}_n$ has generators 
\begin{align*}
E_{i,k},\ F_{i,k},\ H_{i,r},\ K_i^{\pm1},\ q^{\pm c} \quad  
(i\in \Z/n\Z,\ k\in\Z,\  
r\in\Z/\{0\}).
\end{align*}
In order to write down the defining relations, 
introduce the generating series
\begin{align*}
E_i(z) =\sum_{k\in \Z}E_{i,k}z^{-k}, \quad 
F_i(z) =\sum_{k\in\Z}F_{i,k}z^{-k}, \quad
K_i^{\pm}(z) = K_i^{\pm 1} \exp(\pm(q-q^{-1})\sum_{r=1}^\infty H_{i,\pm r}z^{\mp r})\,.
\end{align*}
Define further $g_{i,j}(z,w)$ by
\begin{align*}
n\ge 3&:\quad
g_{i,j}(z,w)=\begin{cases}
              z-q_2w & (i\equiv j),\\
	      z-q_1w & (i\equiv j-1),\\
	      z-q_3w & (i\equiv j+1),\\
              z-w & (i\not\equiv j,j\pm1).\\
	     \end{cases}\\
n=2&:\quad 
 g_{i,j}(z,w)=\begin{cases}
	      z-q_2w & (i\equiv j),\\
              (z-q_1w)(z-q_3w)& (i\not\equiv j).
	     \end{cases}\\
n=1&:\quad 
 g_{0,0}(z,w)=(z-q_1w)(z-q_2w)(z-q_3w).
\end{align*}
and 
\begin{align*}
d_{i,j}=
\begin{cases}
d^{\mp 1} &  (i\equiv j\mp1, n\ge 3),   \\
-1& (i\not\equiv j, n=2), \\
1 & (\text{otherwise}).\\
\end{cases}
\end{align*}
Notation being as above, the defining relations for 
 $\mathcal{E}_n$
\footnote{
We have slightly changed the notation from \cite{FJMM2}. 
The generators $K^{\pm}_i(z)$, $H_{i,r}$ here 
correspond to $K^{\pm}_i(q^{-c/2}z)$, $q^{rc/2}H_{i,r}$ there respectively.
\textcolor{4.2}{For $n=1$, see Remark 2 in Section 2.2}
}
are as follows:
\begin{gather*}
K_i K^{-1}_i = K^{-1}_i K_i=1, \\
\text{$q^{\pm c}$ are central},\quad 
q^{c}q^{-c}=q^{-c}q^{c}=1\,,\\
K^\pm_i(z)K^\pm_j (w) = K^\pm_j(w)K^\pm_i (z), 
\\
\frac{g_{i,j}(q^{-c}z,w)}{g_{i,j}(q^cz,w)}
K^-_i(z)K^+_j (w) 
=
\frac{g_{j,i}(w,q^{-c}z)}{g_{j,i}(w,q^cz)}
K^+_j(w)K^-_i (z),
\\
d_{i,j}g_{i,j}(z,w)K_i^\pm(q^{(1\mp1)c/2}z)E_j(w)+g_{j,i}(w,z)E_j(w)K_i^\pm(q^{(1\mp1) c/2}z)=0,
\\
d_{j,i}g_{j,i}(w,z)K_i^\pm(q^{(1\pm1)c/2}z)F_j(w)+g_{i,j}(z,w)F_j(w)K_i^\pm(q^{(1\pm1) c/2}z)=0\,,
\end{gather*}
\begin{gather*}
[E_i(z),F_j(w)]=\frac{\delta_{i,j}}{q-q^{-1}}
(\delta\bigl(q^c\frac{w}{z}\bigr)K_i^+(z)
-\delta\bigl(q^c\frac{z}{w}\bigr)K_i^-(w)),\\
d_{i,j}g_{i,j}(z,w)E_i(z)E_j(w)+g_{j,i}(w,z)E_j(w)E_i(z)=0, \\
d_{j,i}g_{j,i}(w,z)F_i(z)F_j(w)+g_{i,j}(z,w)F_j(w)F_i(z)=0.\\
\end{gather*}
In addition we impose the Serre relations as follows.
We use the notation $[A,B]_p=AB-pBA$.
\medskip

\noindent  For $n\ge 3$,
\be
&[E_i(z),E_j(w)]=0, \quad [F_i(z),F_j(w)]=0\quad (i\neq j,j\pm 1),\\
&\mathop{\mathrm{Sym}}_{z_1,z_2}
[E_i(z_1),[E_i(z_2),E_{i\pm1}(w)]_q]_{q^{-1}}=0\,,
\\
&\mathop{\mathrm{Sym}}_{z_1,z_2}
[F_i(z_1),[F_i(z_2),F_{i\pm1}(w)]_q]_{q^{-1}}=0\,.
\en
\medskip

\noindent  For $n=2$, $i\not \equiv j$,
\bea
\mathop{\mathrm{Sym}}_{z_1,z_2,z_3}
\bigl[E_i(z_1),\bigl[E_i(z_2),\bigl[E_i(z_3),E_j(w)\bigr]_{q^2}\bigr]\bigr]_{q^{-2}}=0\,,
\label{quartic1}\\
\mathop{\mathrm{Sym}}_{z_1,z_2,z_3}
\bigl[F_i(z_1),\bigl[F_i(z_2),\bigl[F_i(z_3),F_j(w)\bigr]_{q^2}\bigr]\bigr]_{q^{-2}}=0\,.
\label{quartic2}
\ena

\medskip

\noindent  For $n=1$,
\begin{gather*}
\mathop{\mathrm{Sym}}_{z_1,z_2,z_3}z_2z_3^{-1}[E_0(z_1),[E_0(z_2),E_0(z_3)]]=0\,,\\
\mathop{\mathrm{Sym}}_{z_1,z_2,z_3}z_2z_3^{-1}[F_0(z_1),[F_0(z_2),F_0(z_3)]]=0\,.
\end{gather*}
In the above, $\mathrm{Sym}_{z_1,\cdots,z_s}$ stands for the symmetrization in $z_1,\cdots,z_s$. 
\medskip

\subsection{Some technical points} In this subsection we give a few remarks about the relations in $\mc E_n$,
which are important for this work.

It is convenient to rewrite the
relations involving $K^{\pm}_i(z)$ in terms of the generators $\{H_{i,r}\}$. 
Let $[x]=(q^x-q^{-x})/(q-q^{-1})$. 

First of all, we have
\begin{align*}
K_iE_j(z)K^{-1}_i=q^{a_{i,j}}E_j(z),\ K_iF_j(z)K^{-1}_i=q^{-a_{i,j}}F_j(z).
\end{align*}
The other relations are as follows.

For $n\ge 3$, 
\begin{align*}
&[H_{i,r},E_j(z)]= \frac{[r a_{i,j}]}{r} d^{-r m_{i,j}}
q^{(r-|r|)c/2}\,z^r E_j(z)\,,\\
&[H_{i,r},F_j(z)]=-\frac{[r a_{i,j}]}{r} d^{-r m_{i,j}}
q^{(r+|r|)c/2}\,z^r F_j(z)\,,\\
&[H_{i,r},H_{j,s}]=\delta_{r+s,0}\frac{[r a_{i,j}]}{r}
\frac{q^{rc}-q^{-rc}}{q-q^{-1}}  
d^{-r m_{i,j}}\,.\end{align*}

For $n=2$, 
\begin{align*}
&[H_{i,r},E_j(z)]=a_{i,j}(r) z^r E_j(z)q^{(r-|r|)c/2}\,,
\\
&[H_{i,r},F_j(z)]=-a_{i,j}(r)z^r F_j(z)q^{(r+|r|)c/2}\,,
\\
&[H_{i,r},H_{j,s}]=
\delta_{r+s,0}\,a_{i,j}(r)\frac{q^{rc}-q^{-rc}}{q-q^{-1}},
\end{align*}
where $a_{i,i}(r)=[r](q^r+q^{-r})/r$, 
 $a_{i,j}(r)=-[r](d^r+d^{-r})/r$ ($i\neq j$).
 
For $n=1$, 
\begin{align*}
&[H_{0,r},E_0(z)]=z^rb(r) E_0(z)q^{(r-|r|)c/2},
\\
&[H_{0,r},F_0(z)]=-z^rb(r) F_0(z)q^{(r+|r|)c/2},
\\
&[H_{0,r},H_{0,s}]=\delta_{r+s,0}\,b(r)
\frac{q^{rc}-q^{-rc}}{q-q^{-1}},
\end{align*}
where $b(r)=[r](q^r+q^{-r}-d^r-d^{-r})/r$.
\medskip

The following elements of $\mathcal{E}_n$ are central, 
\begin{align*}
\kappa= K_0\cdots K_{n-1}\,,
\quad   q^{c}\,.
\end{align*}

The algebra $\mathcal{E}_n$ is $\Z^{n}\times \Z$-graded by the 
degree assignment
\bea\label{grading}
&\mathrm{deg}\,E_{i,k}=(1_i,k)\,,
\quad 
\mathrm{deg}\,F_{i,k}=(-1_i,k)\,,
\quad
\mathrm{deg}\,H_{i,r}=(0,r)\,,
\\
&
\mathrm{deg}\,K_i=\mathrm{deg}\,q^{c}=(0,0), \notag 
\ena
where $1_i=(0,\cdots,\overset{i-th}{1},\cdots,0)\in \Z^n$.
For a homogeneous element $x\in\mathcal{E}_n$ with 
$\mathrm{deg}\,x=(d_0,\cdots,d_{n-1},k)$, we set 
$\mathrm{pdeg}\,x=\sum_{i=0}^{n-1}d_i$ and call it the 
{\it principal degree}. 
We have 
\bea\label{princ deg}
&\mathrm{pdeg}\,E_{i,k}=1\,,
\quad
\mathrm{pdeg}\,F_{i,k}=-1\,,
\quad
\mathrm{pdeg}\,H_{i,r}=0\,.
\ena
\mdf{
In Section 4 we use the classical weight of a homogeneous element
\bea\label{weight}
{\rm cweight}\, x=\sum_{i=1}^{n-1}(d_i-d_0)\al_i
\ena
where $\al_i$ $(i=1,\ldots,n-1)$ are the $\mathfrak{sl}_n$ roots.
}

The algebra $\mathcal{E}_n$ has also a formal coproduct
\begin{align*}
&\Delta E_i(z)=E_i(z)\otimes 1+K_i^-(C_1 z)\otimes E_i(C_1 z)\,,
\\
&\Delta F_i(z)=F_i(C_2 z)\otimes K_i^+(C_2 z)+1\otimes F_i(z)\,,
\\
&\Delta K^{+}_i(z)=K_i^+(z)\otimes K^+_i(C_1^{-1} z)\,,
\\
&\Delta K^{-}_i(z)=K_i^-(C_2^{-1} z)\otimes K^-_i(z)\,,
\\
&\Delta\, q^{c}=q^{c}\otimes q^{c}\,,
\end{align*}
where we have set $C_1=q^{c}\otimes 1$
and $C_2=1\otimes q^{c}$. 
Since the right hand side contains an infinite sum of generators, 
these formulas are not a coproduct in the usual sense.
Nevertheless for a certain class of modules 
they can be used to define a module structure on 
tensor products. 
For the
details see \cite{FJMM1},\cite{FJMM2}.

In the sequel, when necessary we shall exhibit the dependence on $q_i$ explicitly 
and write $\mathcal{E}_n$ as $\mathcal{E}_n(q_1,q_2,q_3)$.  
\medskip

\noindent{\it Remark 1.}\quad
The definition of the quantum toroidal algebra with $n\geq 3$ is due to \cite{GKV}. 
Our presentation of $\mathcal{E}_n$ ($n\ge 3$) 
follows closely the one given in \cite{TU}. 

To the authors' knowledge, 
the algebra $\mathcal{E}_1$ has been 
introduced for the first time in \cite{BS}, 
where it was termed the elliptic Hall algebra.  
Subsequently the same algebra has been rediscovered by other authors. 
In \cite{M07} it was called a $(q,\gamma)$ analog of 
 $\mathcal{W}_{1+\infty}$, 
and in \cite{FHHSY} it was called Ding-Iohara algebra.  
In \cite{FFJMM1}, \cite{FFJMM2} we called it 
``quantum continuous $\mathfrak{gl}_\infty$''.
\medskip

\noindent{\it Remark 2.}\quad
In our previous paper \cite{FJMM1} we have used an algebra which is an extension of
$\mathcal{E}_1$ by an additional central element.  
The correspondence of the notation in \cite{FJMM1} and the present
paper is 
$e(z)=(1/(1-q_1)) E_0(z)$, $f(z)=-(q^{-1}/(1-q_3))\mathfrak{c} F_0(z)$, 
$\psi^{\pm}(z)=\mathfrak{c} K_0^{\pm}(q^{c}z)$,
where 
$\mathfrak{c}$ is an extra central element. 
In the generators $e(z),f(z),\psi^\pm(z)$, the defining relations are completely symmetric in 
the parameters $q_1,q_2,q_3$. 
Hence $\mathcal{E}_1(q_{\pi(1)},q_{\pi(2)},q_{\pi(3)})=\mathcal{E}_1(q_1,q_2,q_3)$
for any permutation $\pi$ of $\{1,2,3\}$. 
In contrast, in the case $n\ge 2$ the $q_1\leftrightarrow q_3$ symmetry 
holds true, the map is given by \Ref{iota} below, but $q_2$ plays a distinguished role. 
\medskip

\noindent{\it Remark 3.}\quad
We have not been able to find the Serre relations for $\mathcal{E}_2$ 
in the literature, 
except \cite{M01} where the special case $d=q$ is treated. Our quartic relations are similar to that of \cite{M01}.

\medskip

 For $\mathcal E_2$ we also have cubic relations inspired by the consideration of 
`fused currents' which will be discussed in Section \ref{fused current} and Theorem \ref{sub alg}. 
These cubic relations are not discussed in \cite{M01}. As we show these cubic relations are equivalent to the quartic Serre relations \Ref{quartic1}, \Ref{quartic2} in the presence of quadratic relations.

\begin{lem}\label{cubic lem} In $\mc E_2$ we have the following cubic relations:
\begin{gather*}
\mathop{\mathrm{Sym}}_{z_1,z_2}\Bigl[
q_1(z_1-q_3w)(z_2-q_3w)E_i(z_1)E_i(z_2)E_j(w)-
(1+q_2^{-1})(z_1-q_3w)(q_1z_2-w)E_i(z_1)E_j(w)E_i(z_2)\\
+q_3(q_1z_1-w)(q_1z_2-w)E_j(w)E_i(z_1)E_i(z_2)\Bigr]=0\,,\\
\mathop{\mathrm{Sym}}_{z_1,z_2}\Bigl[
q_3(q_1z_1-w)(q_1z_2-w)F_i(z_1)F_i(z_2)F_j(w)-
(1+q_2^{-1})(q_1z_1-w)(z_2-q_3w)F_i(z_1)F_j(w)F_i(z_2)\\
+q_1(z_1-q_3w)(z_2-q_3w)F_j(w)F_i(z_1)F_i(z_2)\Bigr]=0\,,
\end{gather*}
and the relations obtained by interchanging $q_1$ with $q_3$. 
\end{lem}
\begin{proof}
Starting from the special case of the quartic relation
\begin{align*}
[E_{i,k},[E_{i,k},[E_{i,k},E_{j,l}]_{q^2}]]_{q^{-2}}=0\,,
\end{align*}
we compute the commutator with $F_{i,\pm1-k}$. 
The result is 
\begin{align*}
&(1+q_2^{-1})[E_{i,k},[E_{i,k\pm1},E_{j,l}]_{q^{-2}}]
=(q_1+q_3)[[E_{j,l\pm1},E_{i,k}]_{q^{-2}},E_{i,k}]\,.
\end{align*}
Taking commutators with $H_{i,r}$ we obtain
\begin{align*}
 \mathrm{Sym}_{k_1,k_2}
\left((1+q_2^{-1})
[E_{i,k_1},[E_{i,k_2\pm1},E_{j,l}]_{q^{-2}}]
-(q_1+q_3)[[E_{j,l\pm1},E_{i,k_1}]_{q^{-2}},E_{i,k_2}]
\right)=0\,,
\end{align*}
or in current form
\begin{align*}
 \mathrm{Sym}_{z_1,z_2}
\left((1+q_2^{-1})z_2^{\pm1}
[E_i(z_1),[E_{i}(z_2),E_{j}(w)]_{q^{-2}}]
-(q_1+q_3)w^{\pm1}[[E_{j}(w),E_{i}(z_1)]_{q^{-2}},E_{i}(z_2)]
\right)=0\,.
\end{align*}
Modulo the quadratic relation 
$\mathrm{Sym}_{z,w}(z-q^2w)E_{i}(z)E_i(w)=0$, 
these equations are equivalent to the first identity in the lemma.  
\end{proof}

In fact the quadratic relations
\begin{align*}
 (z-q_1w)(z-q_3w)E_i(z)E_j(w)=(w-q_1z)(w-q_3z)E_j(w)E_i(z).
\end{align*}
with $i\neq j$ also follow from the quartic relations.

\medskip

On the other hand,  the quartic Serre relations ,
are consequences of the quadratic and cubic relations, see 
the part of Section \ref{cor func} concerning the Serre relations.

\subsection{Horizontal and vertical subalgebras}\label{embedding}
In this subsection we describe subalgebras of $\E_n$ isomorphic to 
the quantum affine algebras $U_q(\widehat{\mathfrak{sl}}_n)$ and $U_q(\widehat{\mathfrak{gl}}_n)$
\mdf{for $n\geq2$.}


The algebra 
$U_q(\widehat{\mathfrak{sl}}_n)$ 
has a presentation in terms of the Chevalley generators 
$\{e_i,f_i,t_i^{\pm1}\}$, $0\le i\le n-1$, as follows.
\begin{align*}
&t_it_j=t_jt_i,\quad t_it_i^{-1}=t_i^{-1}t_i=1,\\ 
&t_ie_jt_i^{-1}=q^{a_{i,j}}e_j, \quad t_if_jt_i^{-1}=q^{-a_{i,j}}f_j, 
\\
&[e_i,f_j]=\delta_{i,j}\frac{t_i-t_i^{-1}}{q-q^{-1}},
\\
&[e_i,e_j]=0,\quad [f_i,f_j]=0\quad (\text{if $a_{i,j}=0$}),
\\
&[e_i,[e_i,e_j]_{q^{-1}}]_q=0,
\quad [f_i,[f_i,f_j]_{q}]_{q^{-1}}=0
\quad (\text{if $a_{i,j}=-1$})\,.
\end{align*}
When $n=2$, the last line is to be replaced by
\begin{align*}
 &[e_i,[e_i,[e_i,e_j]_{q^{-2}}]_1]_{q^2}=0,
\quad [f_i,[f_i,[f_i,f_j]_{q^{2}}]_1]_{q^{-2}}=0
\quad (i\neq j).
\end{align*}

Alternatively, $U_q(\widehat{\mathfrak{sl}}_n)$ has a presentation in terms of the Drinfeld generators 
$\{x^{\pm}_{i,l}, h_{i,r}, k_i^{\pm1},q^{\pm c}\}$, 
$1\le i\le n-1$, $l\in\Z$, $r\in\Z\backslash\{0\}$
with the relations
\begin{align*}
& \mdf{k_ik_i^{-1}=k_i^{-1}k_i=1,}\quad q^{c}q^{-c}=q^{-c}q^{c}=1,\\
&\text{$q^{\pm c}$ \mdf{are} central},\quad [k_i,k_j]=[k_i,h_{j,r}]=0\,,
\\
&[h_{i,r},h_{j,s}]=\delta_{r+s,0}\frac{[r a_{i,j}]}{r}\frac{q^{rc}-q^{-rc}}{q-q^{-1}}\,,
\\
&k_ix^{\pm}_{j,l}k_i^{-1}=q^{\pm a_{i,j}}x^\pm_{j,l},\quad
[h_{i,r},x^{\pm}_{j,l}]=\pm\frac{[ra_{i,j}]}{r}q^{(r\mp|r|)c/2}x^\pm_{j,l+r},\\
&[x^+_{i,k},x^-_{j,l}]=\frac{\delta_{i,j}}{q-q^{-1}}\left(q^{-lc}\phi^+_{i,k+l}
-q^{-kc}\phi^-_{i,k+l}\right)\,,
\\
&[x^\pm_{i,k+1},x^{\pm}_{i,l}]_{q^{\pm2}}  +[x^\pm_{i,l+1},x^{\pm}_{i,k}]_{q^{\pm2}}  =0,
\\
&[x^\pm_{i,k},x^\pm_{j,l}]=0\quad (\text{if $a_{i,j}=0$}),\\
&[x^\pm_{i,k+1},x^\pm_{j,l}]_{q^{\mp1}}+[x^\pm_{j,l+1},x^\pm_{i,k}]_{q^{\mp1}}=0\quad
(\text{if $a_{i,j}=-1$}),
\\
&\mathrm{Sym}_{k_1,k_2}
[x^{\pm}_{i,k_1},[x^{\pm}_{i,k_2},x^{\pm}_{j,l}]_{q^{-1}}]_{q}=0
\quad (\text{if $a_{i,j}=-1$})\,.
\end{align*}
In the above we have set
\begin{align*}
\sum_{\pm k\ge0}\phi^{\pm}_{i,k}z^{-k}=k_i^{\pm1}
\exp\Bigl(\pm(q-q^{-1})\sum_{\pm r>0}h_{i,r}z^{-r}\Bigr)\,.
\end{align*}
We choose the correspondence of these two generators as follows.
\begin{align*}
&e_i=x^+_{i,0},\quad f_i=x^-_{i,0},\quad t_i=k_i\quad (1\le i\le n-1),
\quad t_0t_1\cdots t_{n-1}=q^c,
\\
&e_0=q^c(k_1\cdots k_{n-1})^{-1}[\cdots [x^-_{1,1},x^-_{2,0}]_q,\cdots,x^-_{n-1,0}]_q\,,
\\
&f_0=[x^+_{n-1,0},\cdots [x^+_{2,0},x^+_{1,-1}]_{q^{-1}},\cdots]_{q^{-1}}
k_1\cdots k_{n-1}q^{-c}\,.
\end{align*}
\mdf{In order to express the Drinfeld generators in terms of
the Chevalley generators, it is useful to have the formulas:
\begin{align*}
&x^-_{i,1}=(-1)^{i-1}(t_0\cdots t_{i-1}t_{i+1},\cdots t_{n-1})^{-1}
[\cdots[e_0,e_{n-1}]_{q^{-1}}\cdots,e_{i+1}]_{q^{-1}},e_1]_{q^{-1}}\cdots,e_{i-1}]_{q^{-1}},\\
&h_{i,1}=(-1)^i[\cdots[e_0,e_{n-1}]_{q^{-1}}\cdots,e_{i+1}]_{q^{-1}},e_1]_{q^{-1}}
\cdots,e_{i-1}]_{q^{-1}},e_i]_{q^{-2}},\\
&x^+_{i,-1}=(-1)^{i-1}[f_{i-1},\cdots[f_1,[f_{i+1},\cdots[f_{n-1},f_0]_q\cdots]_q
t_0\cdots t_{i-1}t_{i+1},\cdots t_{n-1},\\
&h_{i,-1}=(-1)^i[f_i,[f_{i-1},\cdots[f_1,[f_{i+1},\cdots[f_{n-1},f_0]_q\cdots]_q]_{q^2}.
\end{align*}
}

A characteristic feature of the algebra $\mathcal{E}_n$  is that for $n\ge 2$ it admits 
two different embeddings of $U_q(\widehat{\mathfrak{sl}}_n)$, 
\begin{align*}
 h,v: U_q(\widehat{\mathfrak{sl}}_n)\longrightarrow 
\mathcal{E}_n.
\end{align*}
The embedding $h$ is defined in terms of the Chevalley generators, 
\begin{align*}
h:~ e_i\mapsto E_{i,0},\quad f_i\mapsto F_{i,0},\quad
t_i\mapsto K_{i}\quad (0\le i\le n-1).
\end{align*}
The embedding $v$ is defined in terms of the Drinfeld generators, 
\begin{align*}
v&:~ x^+_{i,k}\mapsto  d^{ik}E_{i,k},\quad 
 x^-_{i,k}\mapsto d^{ik}F_{i,k},\quad 
k_i\mapsto K_{i},\quad 
h_{i,r}\mapsto d^{ir}H_{i,r},
 \quad q^{c}\mapsto q^{c}\\
&(1\le i\le n-1, k\in\Z, r\in\Z\backslash\{0\}).
\end{align*}
We call $h$  the {\it horizontal embedding}, and its image 
$h\bigl(U_q\widehat{\mathfrak{sl}}_n\bigr)$ 
the {\it horizontal subalgebra} of $\mathcal{E}_n$. 
Similarly we call $v$ the {\it vertical embedding}, and its image
$v\bigl(U_q\widehat{\mathfrak{sl}}_n\bigr)$ the {\it vertical subalgebra}
of $\mathcal{E}_n$. We denote the horizontal and vertical subalgebras by
$U^{hor}_q(\widehat{\mathfrak{sl}}_n)$ and $U^{ver}_q(\widehat{\mathfrak{sl}}_n)$ respectively.
\mdf{
Note that if $x\in U^{hor}_q(\widehat{\mathfrak{gl}}_n)$ then $\deg x\in\Z^n\times\{0\}$, and if
$x\in U^{ver}_q(\widehat{\mathfrak{gl}}_n)$ then $\deg x\in\{0\}\times\Z^{n-1}\times\Z$.
\textcolor{jimbo}{Note also that}
$\E_n$ is generated by the union of
$U^{hor}_q(\widehat{\mathfrak{sl}}_n)$ and $U^{ver}_q(\widehat{\mathfrak{sl}}_n)$.
}

As it is pointed out in \cite{FJMM2}, there are also 
Heisenberg subalgebras commuting with these subalgebras. 
For each $r\neq 0$, let $\{c_{i,r}\}_{i=0}^{n-1}$ be a non-trivial solution of the equation
\begin{align*}
\sum_{i=0}^{n-1}c_{i,r}[r a_{i,j}] d^{-r m_{i,j}}=0
\quad (j=1,\ldots,n-1).
\end{align*}
Let $\mathfrak{a}^{ver}$ be the subalgebra of $\mathcal{E}_n$ 
generated by $H^{ver}_r=\sum_{i=0}^{n-1}c_{i,r}H_{i,r}$, $r\in\Z_{\neq0}$. 
Clearly $\mathfrak{a}^{ver}$ is a Heisenberg subalgebra with central element $q^c$ which
commutes with $U_q^{ver}\widehat{\mathfrak{sl}}_n$. 
We call the subalgebra generated by these two 
the {\it vertical quantum affine $\mathfrak{gl}_n$} and denote it by
$U^{ver}_q(\widehat{\mathfrak{gl}}_n)$. 

Similarly there exists 
a Heisenberg subalgebra $\mathfrak{a}^{hor}$ which commutes with
$U_q^{hor}\widehat{\mathfrak{sl}}_n$. In terms of the automorphism $\theta$ to be given 
in Theorem \ref{theta} below, we have
$\mathfrak{a}^{hor}
=\theta^{-1}\bigl(\mathfrak{a}^{ver}\bigr)$.

We call the subalgebra generated by these two 
the {\it horizontal quantum affine $\mathfrak{gl}_n$} and denote it by
$U^{hor}_q(\widehat{\mathfrak{gl}}_n)$. 
The central element $\kappa=h(q^c)$ belongs to the horizontal 
subalgebra, 
while $q^c=v(q^c)$ belongs to the vertical subalgebra. 


\subsection{Automorphisms}\label{aut sec}
The algebra $\mathcal{E}_n(q_1,q_2,q_3)$ allows for various symmetries. 

First, there exist automorphisms of algebras 
\begin{align*}
 \tau,s_a,\chi_j:\mathcal{E}_n(q_1,q_2,q_3)\to\mathcal{E}_n(q_1,q_2,q_3)\,,
\end{align*}
where $a\in\C^{\times}$ and $0\le j\le n-1$, such that
\begin{align}
&\tau:
E_i(z)\mapsto E_{i+1}(z),\quad  
F_i(z)\mapsto F_{i+1}(z),\quad  
K^\pm_i(z)\,\mapsto K_{i+1}^{\pm}(z),
\label{tau}\\
&s_a:
E_i(z)\mapsto E_{i}(az), 
\quad F_i(z)\mapsto F_{i}(az), 
\quad K^{\pm}_i(z)\mapsto K^{\pm}_i(az)\,,
\label{sa}\\
&\chi_j:E_i(z)\mapsto E_{i}(z)z^{-\delta_{i,j}},
\quad F_i(z)\mapsto F_{i}(z)z^{\delta_{i,j}},
\quad K^\pm_i(z)\mapsto q^{\mp \delta_{i,j}c} K^\pm_i(z)\,,\label{chij}
\end{align}
and such that all these maps send $q^c$ to itself. 

We have $\tau^n=id$. 

In addition, there exists an isomorphism of algebras 
\begin{align*}
\iota:\E_n(q_1,q_2,q_3)\to \E_n(q_3,q_2,q_1),
\end{align*}
given by 
\begin{align}
\iota:E_i(z)\mapsto E_{n-i}(z), \quad
F_i(z)\mapsto F_{n-i}(z),\quad  
K_i^{\pm}(z)\mapsto K_{n-i}^{\pm}(z)\,,
\label{iota}
\end{align}
and  $\iota(q^{c})=q^c$. 

Of particular importance is the existence of an automorphism which 
exchanges the horizontal subalgebra 
$U^{hor}_q(\widehat{\mathfrak{gl}}_n)$ 
and the vertical subalgebra
$U^{ver}_q(\widehat{\mathfrak{gl}}_n)$. 

Let $\sigma,\eta'$ be anti-automorphisms of 
$U_q(\widehat{\mathfrak{sl}}_n)$ given by
\begin{align*}
&\sigma:~ e_i\mapsto e_i,\quad f_i\mapsto f_i,\quad t_i\mapsto t_i^{-1}\,,
\\ 
&\eta': ~x^{\pm}_{i,k}\mapsto x^{\pm}_{i,-k},
\quad h_{i,r}\mapsto -q^{r c}h_{i,-r},\quad
k_i\mapsto k_i^{-1}, \quad
q^{c}\mapsto q^{c}. 
\end{align*}

\begin{thm}\label{theta}\cite{M99},\cite{M01}\quad
Let $n\ge 2$. 
There exists a unique automorphism $\theta$ of $\E_n$
such that\footnote{Our $\theta$ here is $\psi$ of \cite{M99}.} 
\begin{align*}
\theta\circ v= h\,,
\quad
\theta\circ h = v \circ \eta'\circ \sigma\,.
\end{align*}
We have $\theta(q^c)=\kappa$ and $\theta(\kappa)=q^{-c}$.
\end{thm}

\begin{thm}\label{theta2}\cite{BS},\cite{M07}\quad
There exists a unique automorphism  $\theta$ of $\E_1$ such that
\footnote{Our $\theta$ is 
$\psi$ of \cite{M07} followed by the automorphism
$E_0(z)\mapsto q^{-c}E_0(z)$, $F_0(z)\mapsto q^{c}F_0(z)$, 
$K_0^\pm(z)\mapsto K_0^\pm(z)$, $q^c\mapsto q^c$. 
Unlike $\psi$, $\theta^4\neq \mathrm{id}$.}
\begin{align*}
&E_{0,0}\mapsto -q^c H_{0,-1}\,,\quad
F_{0,0}\mapsto a^{-1}q^{-c}H_{0,1}\,,\\
&H_{0,1}\mapsto E_{0,0}\,,
\quad 
H_{0,-1}\mapsto -a F_{0,0}\,,
\\
&q^c\mapsto K_{0}\,, \quad K_0\mapsto q^{-c},
\end{align*}
where $a=q(1-q_1)(1-q_3)$.
\end{thm}
\medskip

\noindent{\it Remark.}\quad 
Actually, in the case of $\E_2$, the existence of $\theta$ has been 
proved only in the case $q_1=1$ \cite{M01}. 
\textcolor{jimbo}{
It can be shown that with minor modifications 
the method of  \cite{M01} carries over to the general case. }
\medskip

We shall write 
\bea\label{perp gen}
x^{\perp}=\theta^{-1}(x)  
\quad (x\in \mathcal{E}_n). 
\ena
Then we have 
\bea\label{c perp}
&E^{\perp}_{i,0}=E_{i,0},\quad F^{\perp}_{i,0}=F_{i,0},\quad
K^{\perp}_{i}=K_{i}\quad (1\le i\le n-1),\notag
\\
&(q^c)^{\perp}=\kappa^{-1},\quad \kappa^{\perp}=q^c\,,
\ena
and for $n\geq 2$
\begin{align*}
&E^{\perp}_{0,0}
=d\kappa^{-1}q^{c}
K_{0}[\cdots[F_{1,1},F_{2,0}]_q,\cdots,F_{n-1,0}]_q\,,
\\
&
F^{\perp}_{0,0}
=d^{-1}\kappa q^{-c}
[E_{n-1,0},\cdots,[E_{2,0},E_{1,-1}]_{q^{-1}},\cdots]_{q^{-1}}
K_{0}^{-1}
\,.
\end{align*}

\mdf{
We also have
\begin{align*}
&H_{i,1}^\perp=-(-d)^{-i}\kappa^{-1}[\cdots
[F_{0,0},F_{n-1,0}]_q\cdots,F_{i+1,0}]_q,F_{1,0}]_q\cdots F_{i-1,0}]_q,F_{i,0}]_{q^2},\\
&H_{i,-1}^\perp=-(-d)^i\kappa[E_{i,0},[E_{i-1,0},\cdots[E_{1,0},[E_{i+1,0},\cdots[E_{n-1,0},
E_{0,0}]_{q^{-1}}\cdots]_{q^{-1}}]_{q^{-2}},\\[5pt]
\text{for }&i=1,\ldots,n-1, 
\end{align*}
and for $n\geq2$ we have
\begin{align*}
&F_{0,1}^\perp=(-d)^{-n}F_{0,-1},\\
&H_{0,1}^\perp=-(-d)^{-n+1}\kappa^{-1}[\cdots[F_{1,1},F_{2,0}]_q\cdots F_{n-1,0}]_q,F_{0,-1}]_{q^2},\\
&E_{0,-1}^\perp=(-d)^nE_{0,1},\\
&H_{0,-1}^\perp=-(-d)^{n-1}\kappa
[E_{0,1},[E_{n-1,0}\cdots[E_{2,0},E_{1,-1}]_{q^{-1}}\cdots]_{q^{-1}}]_{q^{-2}}.\\[-10pt]
\end{align*}
}
\mdf{
The following Lemma can be extracted from \textcolor{jimbo}{\cite{M00}:}
\begin{lem}\label{perp deg}
If $x\in\E_n$ has degree $(l,d_1+l,\ldots,d_{n-1}+l,k)$
then $x^\perp=\theta^{-1}(x)$ has degree $(-k,d_1-k,\ldots,d_{n-1}-k,l)$.
\end{lem}
}

\mdf{In particular,} the principal degrees of the `perpendicular generators' are given by
\bea\label{prin deg perp}
&\mathrm{pdeg}\, E^{\perp}_{i,k}=-nk-n\delta_{i,0}+1\,,
\quad
\mathrm{pdeg}\, F^{\perp}_{i,k}=-nk+n\delta_{i,0}-1\,,
\quad
\mathrm{pdeg}\, H^{\perp}_{i,k}=-nk\,.
\ena

Later on we shall use the formulas
\begin{align}
&\theta(H_{i,1})
=
(-d)^{-i}
[[\cdots[[\cdots[E_{0,0},E_{n-1,0}]_{q^{-1}},\cdots,E_{i+1,0}]_{q^{-1}},
E_{1,0}]_{q^{-1}},\cdots,E_{i-1,0}]_{q^{-1}},E_{i,0}]_{q^{-2}}\,,\label{thetaH1}
\\
&
\theta(H_{i,-1})
=
(-d)^{i}
[F_{i,0},[F_{i-1,0},\cdots[F_{1,0},[F_{i+1,0},\cdots
[F_{n-1,0},F_{0,0}]_{q}\cdots]_{q}]_q\cdots]_{q}]_{q^{2}}\,,\nn
\end{align}
where $1\le i\le n-1$;
\begin{align}
&\theta(H_{0,1})
=
(-d)^{-n+1}
[[\cdots[E_{1,1},E_{2,0}]_{q^{-1}},\cdots,
E_{n-1,0}]_{q^{-1}},E_{0,-1}]_{q^{-2}}\,,\label{thetaH01}
\\
&
\theta(H_{0,-1})
=
(-d)^{n-1}
[F_{0,1},[F_{n-1,0},\cdots, 
[F_{2,0},F_{1,-1}]_{q}\cdots]_{q}]_{q^{2}}\,\nn
\end{align}

We recall that for $n=1$ 
\begin{align*}
\theta(H_{0,1})=E_{0,0},\quad 
\theta(H_{0,-1})=-a F_{0,0}\,.
\end{align*}

Let $s_i$, $i=0,\dots,n-1$, denote the Lusztig braid group automorphism
of $U_q\bigl(\widehat{\mathfrak{sl}}_n\bigr)$,  
\begin{align*}
&s_i(e_i)=-f_it_i,\quad s_i(f_i)=-t_i^{-1}e_i,\\
&
\mdf{
s_i(e_j)=
\begin{cases}
[e_i,e_j]_{q^{-1}}&\text{if $n\geq3$};\\
\frac1{[2]}[e_i,[e_i,e_j]_{q^{-2}}]&\text{if $n=2$},\\
\end{cases}\quad
s_i(f_j)=
\begin{cases}
[f_j,f_i]_{q},&\text{if $n\ge3$};\\
\textcolor{jimbo}{\frac1{[2]}[[f_j,f_i]_{q^2},f_i]}
&\text{if $n=2$},
\\
\end{cases}
}
\quad (j\equiv i\pm 1),
\\
&s_i(e_j)=e_j,\quad s_i(f_j)=f_j\quad (j\not\equiv i,i\pm 1),
\\
&s_i(t_j)=t_i^{-a_{i,j}}t_j\,.
\end{align*}

Consider the automorphisms
\begin{align}
&T_{n-1|0}=\theta^{-1}\circ\chi_0\chi_{n-1}^{-1}
\circ\theta\,,
\label{T}
\\
&T=T_{n-1|0}^n\,.\label{shift element}
\end{align}
Since each $\chi_j$ (see \eqref{chij}) preserves the vertical subalgebra, 
$T_{n-1|0}$, $T$ preserve the horizontal subalgebra.
Note also that $T_{n-1|0}$, $T$ restricted to
$\mathfrak{a}^{hor}$ are identity operators.

We shall need the following result.
\begin{lem}\label{Weyl-translation}
We have
\begin{align*}
T^{-1}\circ h=h\circ (s_{n-1}\cdots s_1s_0)^{n-1}.
\end{align*}
\end{lem}
\begin{proof}
Set  $\mc Y_n=\zeta_0\chi_0\chi_{n-1}^{-1}$, 
where $\zeta_0$ is the automorphism of $\mc E_n$ given by 
$E_0(z)\mapsto (-d)^{-n}E_0(z)$, $F_0(z)\mapsto (-d)^{n}F_0(z)$, 
leaving unchanged the rest of the generators.
The lemma follows from Proposition 2 of \cite{M99} 
by choosing $x=\varphi^{-1}(\mc Y_n^{-n})$,
and noting that $\zeta_0\circ v=v$. 
\end{proof}
\mdf{The exists an action of the braid group on any integrable
$U_q(\widehat{\mathfrak{sl}}_n)$ module. Therefore,
there exists an action of $T$ on any integrable 
$U^{hor}_q(\widehat{\mathfrak{gl}}_n)$ module.}

\subsection{Representations}\label{rep sec}
In this subsection we present a family of $\mathcal{E}_n$-modules   
studied in our previous works \cite{FJMM1},\cite{FJMM2}. 
These are
\begin{itemize}
\item the vector representation  $V^{(k)}(u)$,
\item the Fock representation  $\mathcal{F}^{(k)}(u)$,
\item the representation $\mathcal{N}_{\alpha,\beta}^{(k)}(u)$\,.
\end{itemize}
In all cases the central element  $q^{c}$ acts as identity. 
These modules carry a discrete parameter $k\in\Z/n\Z$ 
which we call {\it color}, 
and a continuous parameter $u\in\C^{\times}$ which we call
the {\it evaluation parameter}.
In fact the general case can be obtained as
a twist of the one for $k=0$ and $u=1$
by the automorphisms $\tau$ and $s_a$, given by \eqref{tau} and \eqref{sa} respectively.

First, we recall some terminology about representations. 

An $\E_n$-module 
$V$ is said to have {\it level} \mdf{$\ell$} if the central element $\kappa^{-1}$
acts as the scalar \mdf{$\ell$}. 

Let \mdf{$\boldsymbol{\phi}(z)=(\phi_i^{+}(z), \phi_i^{-}(z))_{i\in\Z/n\Z}$} be a collection of 
formal series \mdf{$\phi^{\pm}_i(z)\in\C[[z^{\mp 1}]]$}.
A vector $v\in V$ is said to have 
{\it weight} \mdf{$\boldsymbol{\phi}(z)$}
if $K^{\pm}_i(z)v=\phi^{\pm}_i(z)v$ holds for all $i\in\Z/n\Z$. 
The module $V$ is {\it weighted} if the action of the 
commuting family of operators $\{K^{\pm}_i(z)\}_{i\in\Z/n\Z}$ 
is diagonalizable in $V$. 
It is said to be {\it tame} if the joint spectrum of this action 
is simple. 

The module $V$ is {\it lowest weight} if it is generated by a weight vector
$v$  such that $F_i(z)v=0$ for all $i\in\Z/n\Z$. 
Such a $v$ is called a {\it lowest weight vector}, and its weight
the lowest weight of $V$.
Given  \mdf{$\boldsymbol{\phi}(z)=(\phi^{+}_i(z),\phi^{-}_i(z))_{i\in\Z/n\Z}$} with 
$\phi^+_i(\infty)\phi^-_i(0)=1$, 
there exists a unique irreducible lowest weight module 
\mdf{$L_{\boldsymbol{\phi}(z)}$} with lowest weight \mdf{$\boldsymbol{\phi}(z)$}.

Let $V=\oplus_{s\in\Z}V_s$ be a module $\Z$-graded by the principal degree. 
(This is the case for all modules we consider in this paper.)
We say that $V$ is {\it quasi-finite} if $\dim V_s<\infty$ for all $s$. 
It is known \cite{M07}, \cite{FJMM2}, that an irreducible lowest weight module
\mdf{$L_{\boldsymbol{\phi}(z)}$} is quasi-finite if and only if, for each $i$,
$\phi_i^\pm(z)$ are expansions of a rational function
$\phi_i(z)$, such that it is regular at $z=0,\infty$ and 
$\phi_i(0)\phi_i(\infty)=1$. If it is the case we say simply that 
the lowest weight is \mdf{$\boldsymbol{\phi}(z)=\bigl(\phi_i(z)\bigr)_{i\in \Z/n\Z}$}. 
\medskip

\noindent {\it Vector representation.}\quad 
The vector representation  $V^{(k)}(u)$ 
has a basis $\{[u]^{(k)}_j\}_{j\in\Z}$.  
For $n\ge 2$, the action of the generators
is explicitly given as follows.
\begin{align*}
&E_i(z)[u]_j^{(k)}
=\begin{cases}\delta(q_1^{j+1}u/z)[u]_{j+1}^{(k)}\,, \qquad &i+j+1\equiv k;\\
0\,,\qquad & i+j+1\not\equiv k;
\end{cases}\notag\\
&F_i(z)[u]_{j+1}^{(k)}
=\begin{cases}\delta(q_1^{j+1}u/z)[u]_j^{(k)}\,, \qquad &i+j+1\equiv k; \\
0\,,\qquad & i+j+1\not\equiv k;
\end{cases}
\\
&K^\pm_i(z)[u]_j^{(k)}
=\begin{cases} \psi(q_1^ju/z)[u]_j^{(k)}\,, \qquad &j+i\equiv k;\\
\psi(q_1^jq_3^{-1}u/z)^{-1} [u]_j^{(k)}\,, \qquad &j+i+1\equiv k;\\
 [u]_j^{(k)},\qquad & \text{otherwise.}
\end{cases} \notag
\end{align*}
Here and after we set
\begin{align*}
\psi(z)=\frac{q-q^{-1}z}{1-z}\,.
\end{align*}
For $n=1$ the formulas read
\begin{align*}
&E_0(z)[u]_j^{(0)}
=\delta(q_1^{j+1}u/z)[u]_{j+1}^{(0)}\,, 
\\
&F_0(z)[u]_{j+1}^{(0)}
=q\frac{1-q_3}{1-q_1^{-1}}
\delta(q_1^{j+1}u/z)[u]_j^{(0)}\,,
\\
&K^\pm_0(z)[u]_j^{(0)}
=\psi(q_1^ju/z)\psi(q_1^jq_3^{-1}u/z)^{-1} 
[u]_j^{(0)} \,.
\end{align*}
\medskip

The vector representation $V^{(k)}(u)$ is an irreducible, tame representation
of level $1$. 
\medskip

\noindent {\it Fock representation.}\quad 
We use the following notation concerning partitions. 
A partition 
$\la=(\la_1,\la_2,\cdots)$ is a sequence of 
non-negative integers $\la_i$
such that only finitely many are nonzero and 
$\la_j\geq\la_{j+1}$ for all $j$. In particular, we denote $\emptyset=(0,0,\dots)$.  
The dual partition $\la'$ is given by $\la'_i=|\{j\ |\ \la_j\geq i\}|$.
We identify a partition $\lambda$
with the set of integer points $(x,y)$ 
on the plane satisfying $1\le x\le \ell(\la)$ and $1\le \mdf{y}\le \lambda_x$,
where $\ell(\lambda)=\la'_1$ 
is the length of $\lambda$.
A pair of natural numbers $(x,y)$ is a {\it convex corner of $\la$} if 
$\la'_{y+1}<\la'_y=x$.
 A pair of natural numbers $(x,y)$ is a {\it concave corner of $\la$} 
if $\la'_y=x-1$ and in addition
 $y=1$ or $\la'_{y-1}>x-1$.
 Let $CC(\la)$ and $CV(\la)$ be the set of concave and convex corners of $\la$ respectively. 

Fixing $k\in\Z/n\Z$, to each point $(x,y)\in\Z^2$ we assign a color $k+x-y\in \Z/n\Z$. 
For $i\in \Z/n\Z$, introduce the set of concave (resp. convex) corners 
of $\la$ of color $i$ as follows. 
\begin{align*}
&CC_i^{(k)}(\la)=\{(x,y)\in CC(\la)\mid k+x-y\equiv i\}\,,
\\ 
&CV_i^{(k)}(\la)=\{(x,y)\in CV(\la)\mid k+x-y\equiv i\}\,.
\end{align*}
Finally, for a partition $\la=(\la_1,\la_2,\dots)$ and $j\in\Z_{\geq 1}$
we write $\la\pm\bs 1_j=(\la_1,\la_2,\dots,\la_j\pm 1,\dots)$. 

The Fock representation $\mathcal{F}^{(k)}(u)$ 
has a basis $\{\ket{\lambda}\}$ indexed by all partitions 
$\la$. 
It is realized as a linear 
subspace of the infinite tensor product of vector representations
\begin{align*}
\F^{(k)}(u)\subset 
V^{(k)}(u)\otimes V^{(k)}(uq_2^{-1})\otimes V^{(k)}(uq_2^{-2})\otimes \dots,
\end{align*}
where 
\begin{align*}
\ket{\la}=[u]_{\la_1-1}^{(k)}\otimes [uq_2^{-1}]_{\la_2-2}^{(k)}\otimes [uq_2^{-2}]_{\la_3-3}^{(k)}\otimes\dots\,.
\end{align*}
Notation being as above, the action of $\E_n$ is given as follows. 

For $i\in \Z/n\Z$, $j\in\Z_{\geq 1}$ such that 
$k+j-\la_j\equiv i+1$, 
set
\begin{align*}
\bra{\la+{\bf 1}_j}E_i(z) \ket{\la}
=\hspace{-10pt}
\prod_{\substack{s=1, \\ k+s-\la_s\equiv i 
}}^{j-1} \hspace{-10pt}
 \psi(q_1^{\la_s-\la_j-1}q_3^{s-j})
\prod_{\substack{s=1,\\  k+s-\la_s\equiv i+1    
}}^{j-1} 
\hspace{-10pt}\psi(q_1^{\la_j-\la_s}q_3^{j-s})
\ \ \delta(q_1^{\la_j}q_3^{j-1}u/z),
\end{align*}
\begin{align*}
\bra{\la}F_i(z) \ket{\la+{\bf 1}_j}
=\hspace{-10pt}\prod_{\substack{s=j+1, \\   
k+s-\la_s\equiv i
}}^{\ell(\lambda)} \hspace{-10pt}
 \psi(q_1^{\la_s-\la_j-1}q_3^{s-j})
\prod_{\substack{s=j+1,\\ k+s-\la_s\equiv i+1 
}}^{\ell(\lambda)+1} \hspace{-10pt}\psi(q_1^{\la_j-\la_s}q_3^{j-s})
\ \ \delta(q_1^{\la_j}q_3^{j-1}u/z).
\end{align*}
Further, for $i\in \Z/n\Z$, set
\begin{align*}
\bra{\la}K_i^{\pm}(z) \ket{\la}=
\hspace{-10pt}\prod_{(x,y)\in\ CV_i^{(k)}(\la)}
\hspace{-6pt}\psi(q_3^xq_1^yq_2u/z)
\hspace{-10pt}\prod_{(x,y)\in\ CC_i^{(k)}(\la)} 
\hspace{-6pt}\psi(q_3^xq_1^yq_2^2u/z)^{-1}.
\end{align*}
We set all other matrix coefficients to be zero.
\mdf{In particular, we see that $E_i(z)$ adds, and $F_i(z)$ removes, a box
of color $i$.}

Here we used the bra-ket notation for the matrix elements 
of the linear operators acting in $\F^{(k)}(u)$ in the basis $\{\ket{\la}\}$.


The Fock representation $\mathcal{F}^{(k)}(u)$
is an irreducible, tame, lowest weight representation of level 
$q$ with lowest weight $(\phi_i(z))$ where
\begin{align*}
\phi_i(z)=\begin{cases}
	   \displaystyle{\frac{q^{-1}-q u/z}{1-u/z}} & (i\equiv k) \\
           1 & (i\not \equiv k)\\
	  \end{cases}\,.
\end{align*}

We remark that the Fock representation was given in \cite{Sa} using vertex operators
(for the perpendicular generators), 
and in \cite{VV2},\cite{STU} using the $q$-wedge spaces. 
The explicit formula for the action of $\E_1$ in the Fock space
was found in \cite{FT}. 

\medskip

\noindent {\it Representation $\mathcal{N}_{\alpha,\beta}^{(p)}(u)$.}\quad 
The representation $\mathcal{N}_{\alpha,\beta}^{(p)}(u)$ is defined as a submodule of a finite 
tensor product of Fock representations. 
Let $\alpha,\beta$ be partitions with $m$ parts, such that $\alpha_m=\beta_m=0$. 
Given a color $p\in\Z/n\Z$ and an evaluation parameter $u\in\C^\times$, set 
\bea\label{colors}
p_i=p-\alpha_i+\beta_i\,,\quad u_i=q_1^{\alpha_i}q_2^{i-1}q_3^{\beta_i}u\,,
\quad i=1,\cdots,m\,.
\ena
Consider the linear subspace 
\bea\label{N in otimes}
\mathcal{N}_{\alpha,\beta}^{(p)}(u)\, \subset\,
\mathcal{F}^{(p_1)}(u_1)\otimes \cdots \otimes \mathcal{F}^{(p_m)}(u_m)
\ena
spanned by vectors $\ket{\la^{(1)}}\otimes \cdots\otimes \ket{\la^{(m)}}$, where 
$\la^{(i)}$ are partitions satisfying the conditions
\bea\label{part cond}
\la^{(i)}_j\ge \la^{(i+1)}_{j+b_i}-a_i,\quad i=1,\cdots,m-1,
\ena
where 
\bea\label{ab}
a_i=\alpha_i-\alpha_{i+1}, \qquad b_i=\beta_i-\beta_{i+1}.
\ena 
Then $\mathcal{N}_{\alpha,\beta}^{(p)}(u)$
is a well-defined $\E_n$-submodule of the tensor product module 
$\mathcal{F}^{(p_1)}(u_1)\otimes \cdots \otimes \mathcal{F}^{(p_m)}(u_m)$.
Moreover it is an irreducible, tame, quasi-finite lowest weight module of level $q^m$
and lowest weight $(\phi_i(z))$, where
\begin{align*}
\phi_i(z)=\prod_{j:p_j\equiv i}\frac{q^{-1}-q u_j/z}{1-u_j/z}\,.
\end{align*}


\section{Construction of subalgebras}\label{subalgebras}
In this section we describe a family of subalgebras of a completion of $\mathcal E_n$. 
These subalgebras satisfy the relations of $\mathcal E_m$ with $m<n$ and act in all lowest weight representations of $\mathcal E_n$. In this section, \mdf{except in Section 3.5,} we always
work in perpendicular generators, see \Ref{perp gen} and \Ref{c perp}.
We use similar notation for the generating series, e.g. $E^\perp_i(z)$, $F^\perp_i(z)$, etc.


\subsection{Definition of current $E_{n-1|0}^\perp(z)$.}\label{fused current}
We give the definition of the "fused" current $E_{n-1|0}^\perp(z)$. 
The construction mimics the extraction 
of the polar term in the operator product of $E_0^\perp(z)$ and of $E_{n-1}^\perp(z)$.

Let first $n\geq 3$. We have the relation
\be
(d^{-1}z-q^{-1}w)E_{n-1}^\perp(z)E_0^\perp(w)=(d^{-1}q^{-1}z-w)E_0^\perp(w)E_{n-1}^\perp(z),
\en
which in components is
\bea\label{comp rel}
E_{n-1,k+1}^\perp E_{0,r}^\perp -q_1E_{n-1,k}^\perp E_{0,r+1}^\perp =q^{-1}E_{0,r}^\perp E_{n-1,k+1}^\perp -dE_{0,r+1}^\perp E_{n-1,k}^\perp .
\ena

We start with the representation-theoretical version. 

\mdf{We define another grading $\deg^\perp$ on $\E_n$ such that:
\bea\label{adm grade}
\deg^\perp E_{i,k}^\perp =\deg^\perp F_{i,k}^\perp =\deg^\perp H_{i,k}^\perp =k.
\ena
}
\mdf{From Lemma \ref{perp deg}, it follows that
$-\deg^\perp x$ is equal to the $0$-th component of $\deg x$.}

Call a graded $\mathcal E_n$ module $V$ {\it admissible} if for every vector $v\in V$ there exists
$N(v)$ such that $gv=0$ for all $g\in\mathcal E_n$ with \mdf{$\deg^\perp g >N(v)$}.
If a representation is admissible
then the formal series $E_i^\perp (z)v$, $F_i^\perp (z)v$, $K_i^{-,\perp}(z)v$ are actually Laurent series in $z$ and the series $K^{+,\perp}_i(z)v$ is a polynomial in $z^{-1}$. Note that in terms of the standard generators, the condition of admissibility is written in terms of the principal grading,
\mdf{(see \Ref{princ deg}, \eqref{prin deg perp})}. In particular,
all lowest weight modules defined in Section \ref{rep sec} are admissible.

Let $V$ be an admissible representation.
Then from \Ref{comp rel}, we obtain that given $k\in\Z$ and $v\in V$, we have
\be
E_{n-1,s+k}^\perp E_{0,-s}^\perp v=q_1^{-1}E_{n-1,s+1+k}^\perp E_{0,-s-1}^\perp v
\en
for large enough $s$.

Define 
\bea\label{main def}
E_{n-1|0,k}^\perp v=q_1^{-s-k}E_{n-1,s+k}^\perp E_{0,-s}^\perp v,
\ena
where $s$ is sufficiently large. Clearly, $E_{n-1|0,k}^\perp $ is a well-defined operator acting in $V$.
We set $E_{n-1|0}^\perp (z)=\sum_{k\in\Z}E_{n-1|0,k}^\perp z^{-k}$.

Equivalently, we can define:
\bea\label{limit def}
E_{n-1|0}^\perp (z)=\lim_{z'\to z}\mdf{(1-\frac{z}{z'})}E_{n-1}^\perp (q_1z')E_0^\perp (z).
\ena

Using \Ref{comp rel} repeatedly, we can also write
\be
q_1^{-s-k}E_{n-1,s+k}^\perp E_{0,-s}^\perp = q_1^{-k}E_{n-1,k}^\perp E_{0,0}^\perp +\sum_{i=0}^{s-1}  q_1^{-1-i-k}(q^{-1}E_{0,-1-i}^\perp E_{n-1,k+1+i}^\perp -dE_{0,-i}^\perp E_{n-1,k+i}^\perp ).
\en

Therefore we can equivalently define:
\be
E_{n-1|0,k}^\perp =q_1^{-k}E_{n-1,k}^\perp E_{0,0}^\perp +\sum_{i=0}^{\infty} q_1^{-1-i-k}(q^{-1}E_{0,-1-i}^\perp E_{n-1,k+1+i}^\perp -dE_{0,-i}^\perp E_{n-1,k+i}^\perp ).
\en
Note that the sum evaluated on any vector $v$ in an admissible representation becomes finite.
This formula shows that the operator $E_{n-1|0,k}^\perp $ belongs to the completion of $\mathcal E_n$ with respect to grading \Ref{adm grade}.

There is one more useful way to write the operators $E_{n-1|0,k}^\perp $, which we use in Section \ref{branching sec}. 
Let $T_{n-1|0}$ be the automorphism of $\mathcal E_n$ given by \Ref{T}.

We have 
\be
T_{n-1|0}E_0^\perp (z)=z^{-1}E_0^\perp (z), \quad T_{n-1|0}E_{n-1}^\perp (z)=zE_{n-1}^\perp (z),\\ T_{n-1|0}F_0^\perp (z)=zF_0^\perp (z), \quad T_{n-1|0}F_{n-1}^\perp (z)=z^{-1}F_{n-1}^\perp (z),
\en
and $T_{n-1|0}$ preserves currents with indexes different from $0$ and $n-1$ as well as $q^c$
\mdf{and} $\kappa$. In particular,
\be
T_{n-1|0}K_i^{\pm,\perp}(z)=\mdf{\kappa^{\mp\delta_{i,n-1}\pm\delta_{i,0}}}K_i^{\pm,\perp}(z). 
\en
Then we clearly have 
\be
E_{n-1|0,k}^\perp =\lim_{s\to \infty} q_1^{-s-k} T_{n-1|0}^s (E_{n-1,k}^\perp E_{0,0}^\perp ).
\en

\medskip

Finally, let us consider the case $n=2$. In this case, we replace the product \Ref{main def} by the following stable combination. 
We set
\be
E_{1|0,k}^{(1),\perp}=q_1^{-s-k}(E_{1,s+k}^\perp E_{0,-s}^\perp -q_3E_{1,s+k-1}^\perp E_{0,-s+1}^\perp ),
\en
where $s$ is sufficiently large. Equivalently we have
\be
E_{1|0}^{(1),\perp}(z)=\lim_{z'\to z}\mdf{(1-\frac{z}{z'})(1-\frac{q_3z}{q_1z'})}E_1^\perp (q_1z')E_0^\perp (z)
\en
and
\be
E_{1|0,k}^{(1),\perp}=\lim_{s\to\infty}q_1^{-s-k}T_{1|0}^s
(E_{1,k}^\perp E_{0,0}^\perp -q_3E_{1,k-1}^\perp E_{0,1}^\perp ).
\en

For $n=2$ we write an extra upper index for the reason explained in Section \ref{others}, see \Ref{E21}, \Ref{E31} below.

\subsection{Other operators}\label{others}
We collect operators obtained by the construction described in Section \ref{fused current}.

Similarly to Section \ref{fused current}, we define a number of other currents. We use formulas of type
\Ref{limit def} keeping in mind that it is always justified by formulas of type \Ref{main def}.

For $n\ge 3$ and $i=0,1,\dots,n-1$, define 
\bea
E_{i|i+1}^\perp (z)&=&\lim_{z'\to z}\mdf{(1-\frac{z}{z'})}E_i^\perp (q_1z')E_{i+1}^\perp (z), \label{Ei} \\
F_{i|i+1}^\perp (z)&=&\lim_{z'\to z}(1-\frac{z'}{z})F_{i+1}^\perp (z)F_i^\perp (q_1z'), \label{Fi} \\
K_{i|i+1}^{\pm,\perp}(z)&=& K_i^{\pm,\perp}(q_1z)K_{i+1}^{\pm,\perp}(z). \label{Ki}
\ena

We also have another family of operators defined by
\bea
E_{i+1|i}^\perp (z)&=&\lim_{z'\to z}\mdf{(1-\frac{z}{z'})}E_{i+1}^\perp (q_3z')E_{i}^\perp (z),\label{Ei'} \\
F_{i+1|i}^\perp (z)&=&\lim_{z'\to z}(1-\frac{z'}{z})F_{i}^\perp (z)F_{i+1}^\perp (q_3z'),\label{Fi'} \\
K_{i+1|i}^{\pm,\perp}(z)&=& K_{i+1}^{\pm,\perp}(q_3z)K_{i}^{\pm,\perp}(z).\label{Ki'}
\ena
All such currents of the same type (e.g. of type $E$) are related to each other by $\mathcal E_n$ automorphisms $\tau$ and $\iota$, see \Ref{tau} and \Ref{iota}, for example, $E_{i|i+1}^\perp (z)=\theta^{-1}\circ\tau^{i+1}\circ\theta\bigl(E_{n-1|0}^\perp (z)\bigr)$
and $E_{i+1|i}^\perp (z)=\theta^{-1}\circ\iota\circ\theta\bigl(E_{n-i-1|n-i}^\perp (z)\bigr)$. 

\medskip

Moreover, we use our construction recursively to obtain new currents.
For example, we set
\be
E_{i|i+1|i+2}^\perp (z)&=&\lim_{z''\to z'}\lim_{z'\to z}\mdf{(1-\frac{z'}{z''})(1-\frac{z}{z'})}E_i^\perp (q_1^2z'')E_{i+1}^\perp (q_1z')E_{i+2}^\perp (z),
\en
or
\be
E_{i|i+1|i}^\perp (z)&=&\lim_{z''\to z'}\lim_{z'\to z}\mdf{(1-\frac{z'}{z''})(1-\frac{z}{z'})}E_i^\perp (q_1q_3z'')E_{i+1}^\perp (q_3z')E_{i}^\perp (z).
\en
One can justify this recursive definition directly, but we defer our discussion to Section \ref{relations}.

Note that our notation $E_{i|i+1|i}^\perp (z)$ contains complete information about the shifts of arguments participating in the corresponding definition. Namely, $i|i+1$ signifies the relative shift of $q_3$ while $i+1|i$ signifies the relative shift of $q_1$. 

For $n=2$, the construction of these currents is quite parallel, but we write an additional upper index to distinguish the formulas in \eqref{Ei}--\eqref{Ki} 
and those in  \eqref{Ei'}--\eqref{Ki'}, e.g., 
\bea
E_{1|0}^{(1),\perp}(z)=\lim_{z'\to z}\mdf{(1-\frac{z}{z'})(1-\frac{q_3z}{q_1z'})}E_1^\perp (q_1z')E_0^\perp (z),\label{E21} \\
E_{1|0}^{(3),\perp}(z)=\lim_{z'\to z}\mdf{(1-\frac{z}{z'})(1-\frac{q_1z}{q_3z'})}E_1^\perp (q_3z')E_0^\perp (z).\label{E31}
\ena
In what follows we also use the notation
\mdf{
\bea\label{n-0}\\
E_{\mdf{n-1||0}}^\perp (z)=E_{n-1|n-2|...|1|0}^\perp (z), \  
F_{\mdf{n-1||0}}^\perp (z)=F_{n-1|n-2|...|1|0}^\perp (z),\  K_{\mdf{n-1||0}}^{\pm,\perp}(z)
=K_{n-1|n-2|...|1|0}^{\pm,\perp}(z).\nn
\ena
}

\subsection{Correlation functions}\label{cor func}
We discuss properties of correlation functions in admissible representations.

Let $V$ be an admissible representation. Choose an arbitrary \mdf{graded} basis.
Choose arbitrary \mdf{basis} vectors $v_1,v_2$. 
We use the standard notation for the correlation functions. For example, we write 
$\langle E_1^\perp (z)E_2^\perp (w) \rangle $ for the matrix coefficient $\langle v_1 |E_1^\perp (z)E_2^\perp (w)|v_2 \rangle$ 
of the operator $E_1^\perp (z)E_2^\perp (w)$. 
We study properties common to all correlation functions, 
and therefore it is not important which admissible representation and which particular matrix element we consider, thus we omit this information from our notation. In all calculations $V,v_1,v_2$ are arbitrary but fixed. 

By the word "current" we mean either $E_i^\perp (z)$, $F_i^\perp (z)$, $K_i^{-,\perp}(z)$ or $K_i^{+,\perp}(z)$. Later we will also use the fused currents.

Algebraic relations between currents translate into properties of correlation functions.

Moreover, if an element $g$ of $\mathcal E_n$ acts by zero in all admissible representations, then $g=0$.
Indeed, such a statement is known to be true in the setting of Lie algebras, so it holds for $U\mc L_n'(d)$,
see Section \ref{classical}. (The proof is analogous to ii) of Theorem 8.4.4 in \cite{D}.)
But since $\mathcal E_n$ is a quantization of $U\mc L_n'(d)$ and all admissible representations of $U\mc L_n'(d)$ quantize to representations of $\mathcal E_n$, this fact is true for $\mc E_n$. Finally, note that admissible representations are graded and if an element $g=\sum_{i=1}^\infty g_i$ with
\mdf{$\deg^\perp g_i=i$} of the completion of $\mathcal E_n$ acts by zero in
an admissible representation,
then all $g_i$ do. Therefore, we have the converse statement:
if all correlation functions satisfy a given property then the currents satisfy an algebraic relation
inside $\mathcal E_n$. 

We now discuss the dictionary between algebraic relations and properties of correlation functions. In the dictionary we consider correlation functions of two or three currents.
The correlation functions of many currents satisfy the same properties for each subset of two or three currents. 

\medskip

{\it Quadratic relations.}
Consider two currents satisfying a quadratic relation. For example, consider $\langle E_1^\perp (z)E_2^\perp (w) \rangle$, $n\geq 3$. Then the quadratic relation for the currents  $E_1^\perp (z)$ and $E_2^\perp (w)$ is
\be
(d^{-1}z-q^{-1}w)E_1^\perp (z)E_2^\perp (w)=(d^{-1}q^{-1}z-w)E_2^\perp (w)E_1^\perp (z).
\en
This is equivalent to the following form of the correlation functions:
\bea\label{corr funct}
\langle E_1^\perp (z)E_2^\perp (w) \rangle =\frac{p(z,w)}{d^{-1}z-q^{-1}w},\qquad  \langle E_2^\perp (w)E_1^\perp (z) \rangle =\frac{p(z,w)}{d^{-1}q^{-1}z-w},
\ena
where $p(z,w)$ is a Laurent polynomial.  

Here the right hand side of the first equation is understood as an expansion in $w/z$, while the right hand side of the second equation as an expansion in $z/w$. Such a convention should be clear and we often do not mention it.

Apart from the poles in the correlation functions which are dictated by the quadratic relations we also have symmetries, when several of the currents are the same. For example, we have 
\bea
\langle E_1^\perp (z_1)E_1^\perp (z_2) \rangle =\frac{p(z_1,z_2)}{z_1-q_2z_2}, 
\ena
where $p(z_1,z_2)$ is a  Laurent polynomial (different from the one in \Ref{corr funct}) such that $p(z_1,z_2)=-p(z_2,z_1)$. In particular, we have $p(z,z)=0$.

\medskip

{\it Commuting currents.} This is an important special case of the quadratic relations. If two currents commute, their correlation function is a Laurent polynomial (no poles). Of course, the converse is not true in general. 
For example, since the currents $K^{\pm,\perp }(w)$ are power series in $w^{\mp1}$, clearly, the correlation functions $\langle E_i^\perp (z)K_i^{+,\perp}(w)\rangle$ and $\langle K_i^{-,\perp}(w)E_i^\perp (z)\rangle$ do not have poles but $\langle E_i^\perp (z)K_i^{-,\perp}(w)\rangle$ and $\langle K_i^{+,\perp}(w)E_i^\perp (z)\rangle$ do. In particular, these currents do not commute.

However, in order to prove that two currents commute, often it is sufficient to check the absence of poles, since we have that the correlation functions in different orders of currents are equal as rational functions, see Section \ref{commute sec}.

\medskip

{\it Serre relations.} The Serre relations (see Section 2.1)
are equivalent to the {\it wheel conditions} for the correlation functions. 
For example,
\begin{align*}
\langle E^\perp_i(z_1)E^\perp_i(z_2)E^\perp_{i\pm1}(w) \rangle
=
\frac{p_{i,\pm}(z_1,z_2,w)}
{(z_1-q_2z_2)(z_1-q^{-1}d^{\pm1}w)(z_2-q^{-1}d^{\pm1}w)}\,,
\end{align*}
where $p_{i,\pm}(z_1,z_2,w)$ 
is a Laurent polynomial skew-symmetric in $z_1,z_2$ 
satisfying the following wheel condition:
\begin{align*}
p_{i,+}(z,q_2z,q_2q_3z)=p_{i,-}(z,q_2z,q_2q_1z)=0\,.
\end{align*}
Note that due to the skew-symmetry, 
we also have $p_{i,+}(z,q_2^{-1}z,q_3z)=p_{i,-}(z,q_2^{-1}z,q_1z)=0$.

This fact is not completely trivial, therefore we sketch the computation for the most difficult case of $n=1$.
Recall the function $g_{00}(z,w)=g(z,w)=(z-q_1w)(z-q_2w)(z-q_3w)$. Then the Serre relation 
\be
\mathop{\on{Sym}}\limits_{z_1,z_2,z_3} \ z_2z_3^{-1}[E_0^\perp (z_1),[E_0^\perp (z_2),E_0^\perp (z_3)]]=0
\en
is equivalent to:
\bea\label{1 Ser}
\hspace{10pt}p(z_1,z_2,z_3)\left(\mathop{\on{Asym}}\limits_{z_1,z_2,z_3}\ (z_2z_3^{-1}(\frac{1}{g(z_1,z_2)g(z_1,z_3)g(z_2,z_3)}+\frac{1}{g(z_1,z_3)g(z_1,z_2)g(z_3,z_2)}\right.\\
\left.
- \frac{1}{g(z_2,z_3)g(z_2,z_1)g(z_3,z_1)}-\frac{1}{g(z_3,z_2)g(z_3,z_1)g(z_2,z_1)}))\right)=0.\notag
\ena
Let us study the result of the anti-symmetrization. 
First, one checks that as a rational function, it is zero. However, all the terms have to be expanded in their own region. We change all the expansions to the region $|z_1| \gg |z_2| \gg |z_3|$ by adding the delta functions.  

The coefficient of a single delta function is obtained from the sum of twelve terms
out of twenty four terms present in \Ref{1 Ser}. One checks that this sum has a zero
at the support of the corresponding delta function.
Therefore single delta functions do not appear.

However, we do have products of two delta functions. The corresponding coefficient is computed from four terms \Ref{1 Ser} and it is non-trivial. For example, we have
$\delta(z_1/(q_1z_2))\delta(q_3z_1/z_3)$ with a non-zero coefficient. Therefore, this product of delta functions is absent if and only if $p(z_1,q_1^{-1}z_1,q_3z_1)=0$.

\medskip

Let us also comment on the Serre relations in the $n=2$ case. As discussed above, we impose the quartic relations, \Ref{quartic1}, \Ref{quartic2}, following \cite{M01}. By Lemma \ref{cubic lem} we also have the cubic relations. These relations
are inspired by Theorem \ref{sub alg} and they are equivalent to the wheel condition for the correlation functions. 

We show here that quartic relations follow from the cubic ones.
Let
\be
\langle E_0^\perp (z_1)E_0^\perp (z_2)E_0^\perp (z_3)E_1^\perp (w)\rangle=\frac{p(z_1,z_2,z_3,w)}{\prod\limits_{i=1}^3(z_i-q_1w)(z_i-q_3w)\prod_{i<j}(z_i-q_2z_j)}.
\en
Then $p(z_1,z_2,z_3,w)$ is a Laurent polynomial which is skew-symmetric in $z_1,z_2,z_3$ and vanishing if $z_i=z_j$, or if $z_2=q_2z_1$ and $w=q_1z_2$, or if $z_2=q_2z_1$ and $w=q_3z_2$.
Relation \Ref{quartic1} is equivalent to 
\be
\lefteqn{p(z_1,z_2,z_3,w)\ \mathop{\on{Asym}}\limits_{z_1,z_2,z_3}\frac{1}{\prod\limits_{i<j}(z_i-q_2z_j)}\Big(
\frac{1}{\prod\limits_{i=1}^3g_{12}(z_i,w)}} \\&&\hspace{40pt}
-\frac{q_2+1+q_2^{-1}}{\prod\limits_{i=1}^2g_{12}(z_i,w)g_{12}(w,z_3)}+ \frac{q_2+1+q_2^{-1}}{g_{12}(z_1,w) \prod\limits_{i=1}^2g_{12}(z_i,w)}-
\frac{1}{ \prod\limits_{i=1}^3g_{12}(w,z_i)}\Big)
=0.
\en
Here as before in the case of $n=2$, $g_{12}(z,w)=(z-q_1w)(z-q_3w)$.
This equality is established in the same way as \Ref{1 Ser} using the vanishing conditions of the Laurent polynomial $p(z_1,z_2,z_3,w)$.
Hence the quartic relations are consequences of the quadratic and cubic relations.
\medskip

{\it Commutators of the $E$ and $F$ currents.} The relation
\be
[E_i^\perp (z),F_i^\perp (w)]=\frac{1}{q-q^{-1}}(\delta(\kappa^{-1}w/z)K_i^{+,\perp }(z)-\delta(\kappa^{-1}z/w)K_i^{-,\perp}(z))
\en
holds if and only if the following formulas are satisfied for all correlation functions:
\be
\langle E_i^\perp (z)F_i^\perp (w)\rangle= \frac{p(z,w)}{(z-\kappa w)(z-\kappa^{-1}w)},\qquad \langle F_i^\perp (w)E_i^\perp (z)\rangle= \frac{p(z,w)}{(z-\kappa w)(z-\kappa^{-1}w)},
\en
where $p(z,w)$ is a Laurent polynomial and
\be
\langle K_i^{+,\perp}(z) \rangle 
=\kappa^{-1} z^{-2}\frac{q-q^{-1}}{\kappa^{-1}-\kappa}p(z,\kappa z)
=-\kappa^{-1}z^{-1}(q-q^{-1})\on{Res}_{w=\kappa z}\langle E_i^\perp (z) F_i^\perp (w)\rangle, 
\\
\mdf{-\langle K_i^{-,\perp}(w) \rangle }
=\kappa^{-1}w^{-2}\frac{q-q^{-1}}{\kappa^{-1}-\kappa} p(\kappa w,w)
=-\kappa^{-1}w^{-1}(q-q^{-1})\on{Res}_{z=\kappa w}\langle E_i^\perp (z) F_i^\perp (w)\rangle.
\en

\subsection{The relations for the fused currents}\label{relations}
In this section we describe the subalgebra generated by fused currents.

From now on we fix $\eta_i$, $i=1,2,3$ such that $q_i=\exp(\eta_i)$. For $x\in\Q$, we set $q_i^x=e^{x\eta_i}$.

Set 
\be
\tilde E_i^\perp (z)=E_i^\perp (q_1^{\frac{i}{n-1}}z), \qquad \tilde F_i^\perp (z)=F_i^\perp (q_1^{\frac{i}{n-1}}z),\qquad \tilde K_i^{\pm,\perp}(z)=K_i^{\pm,\perp}(q_1^{\frac{i}{n-1}}z),
\en 
$i=1,\dots,n-2$, and
set 
\be
\tilde E_0^\perp (z)=E_{n-1|0}^\perp (z),\qquad\tilde F_0^\perp (z)=F_{n-1|0}^\perp (z),\qquad\tilde K_0^{\pm,\perp}(z)=K_{n-1|0}^{\pm,\perp}(z).
\en
Set 
\be
\tilde q_1=q_1\cdot q_1^{\frac{1}{n-1}}, \qquad \tilde q_2=q_2, \qquad \tilde q_3=q_3\cdot q_1^{-\frac{1}{n-1}}.
\en

\begin{thm}\label{sub alg}
The currents $\tilde E_i^\perp (z), \tilde F_i^\perp (z), \tilde K_i^{\pm,\perp}(z)$, $i=0,1,\dots,n-2$, satisfy the relations of the toroidal algebra $\mathcal E_{n-1}(\tilde q_1,\tilde q_2,\tilde q_3)$.
\end{thm}

In the proof of this theorem we use the following simple lemma.
\begin{lem}\label{simple lemma}
Let $f(z_1,z_2,w_1,w_2)$ be a Laurent polynomial, $a,b,c,d$ complex numbers, such that
$f(w,az,bz,cz)=f(dz,az,bz,w)=0.$
Then 
\be
\frac{f(dz,az,bw,cw)}{z-w}\Bigl|_{z=w}=\frac{f(dz,aw,bz,cw)}{w-z}\Bigl|_{z=w}.
\en
\end{lem}
\begin{proof} We have
\be
\frac{f(dz,az,bw,cw)}{z-w}\Bigl|_{z=w}
=\frac{\partial}{\partial z}f(dz,az,bw,cw)\Bigl|_{z=w}=
\frac{\partial}{\partial y}f(dz,ay,bw,cw)\Bigl|_{y=w}.
\en
On the other hand,
\be
\frac{f(dz,aw,bz,cw)}{w-z}\Bigl|_{z=w}
=\frac{\partial}{\partial w}f(dz,aw,bz,cw)\Bigl|_{z=w}=
\frac{\partial}{\partial y}f(dz,ay,bw,cw)\Bigl|_{y=w}.
\en
\end{proof}

\noindent
{\it Proof of Theorem \ref{sub alg}.}
We check the relations for the correlation functions. 
By the construction, the correlation functions of the fused currents are extracted from those of the standard currents in the way similar to obtaining $\langle K_i^{\pm,\perp}(z)\rangle $ from
$\langle E_i^\perp (z)F_i^\perp (w)\rangle$. For example, if
\be
\langle E_{n-1}^\perp (q_1z_1)E_0^\perp (z_2)\rangle =\frac{p(z_1,z_2)}{z_1-z_2}
\en
then
\be
\langle E_{n-1|0}^\perp (z)\rangle =-z^{-1}p(z,z).
\en
It enables us to check its properties. 

All quadratic relations and Serre relations are checked by a straightforward computation. The cases of $n=3$ and $n=2$ are slightly different but not much more difficult. 
For example, let $n=2$, and let us check the relation
\mdf{$\tilde g(z,w)E_{1|0}^{(1),\perp}(z) E_{1|0}^{(1),\perp}(w)
=\tilde g(w,z) E_{1|0}^{(1),\perp}(w) E_{1|0}^{(1),\perp}(z)$, where
\begin{align*}
\tilde g(z,w)=(z-\tilde q_1w)(z-\tilde q_2w)(z-\tilde q_3w)=(z-q_1^2w)(z-q_2w)(z-q_3q_1^{-1}w).
\end{align*}
}
We have
\be
\lefteqn{\langle E_1^\perp (q_1z')E_0^\perp (z)E_1^\perp (q_1w')E_0^\perp (w)\rangle = \frac{z'w'p(q_1z',z,q_1w',w)}{(z'-q_2w')(z-q_2w)}} \\
&&\times \frac{1}{(z'-w)(z'-q_3q_1^{-1}w)(z-q_1^2w')(z-q_1q_3w')(z-z')(q_1z'-q_3z)(w'-w)(q_1w'-q_3w)}.
\en
Then
\be
\langle E_{1|0}^{(1),\perp}(z) E_{1|0}^{(1),\perp}(w) \rangle = \frac{p(q_1z,z,q_1w,w)}{(z-q_2w)^2(z-w)(z-q_3q_1^{-1}w)(z-q_1^2w)(z-q_1q_3w)}.
\en
The factor $(z-q_1q_3w)$ and one factor of \mdf{$(z-q_2w)$} cancel due to the wheel conditions for the Laurent polynomial $p(z',z,w',w)$. Finally the pole $z-w$ is absent due to the skew-symmetry property of $p(z',z,w',w)$. 

The most difficult calculation is the $EF$ relation for the fused current. Here are some details in the case $n\geq 3$.
Consider
\be
\lefteqn{R(z_1,z_2,w_1,w_2):=\langle E_{n-1}^\perp (q_1z_2)E_0^\perp (z_1)F_0^\perp (w_1)F_{n-1}^\perp (q_1w_2)\rangle} \\ &&=
\frac{p(z_1,z_2,w_1,w_2)}{(z_1-z_2)(w_1-w_2)(z_1-\kappa^{-1}w_1)(z_1-\kappa w_1)(z_2-\kappa^{-1}w_2)(z_2-\kappa w_2)}.
\en
Then we have
\be
\langle E_{n-1}^\perp (q_1z_2)K_0^{+,\perp}(z_1)F_{n-1}^\perp (q_1w_2)\rangle=
\frac{\kappa^{-1}z_1^{-2}(q-q^{-1})p(z_1,z_2,\kappa z_1,w_2)}
{(z_1-z_2)(\kappa z_1-w_2)(\kappa^{-1}-\kappa)(z_2-\kappa^{-1}w_2)(z_2-\kappa w_2)}.
\en

But this correlation function does not have a pole at $z_1=z_2$, therefore
the Laurent polynomial $p(z_1,z_2,w_1,w_2)$ satisfies
\bea\label{zero 1}
p(z,z,\kappa z,w_2)=0.
\ena
Similarly, considering the correlation function $\langle E_0^\perp (z_1)F_0^\perp (w_1)K^{+,\perp}_{n-1}(q_1z_2)\rangle$, we obtain
\bea\label{zero 2}
p(z,\kappa^{-1} w,w,w)=0.
\ena
\mdf{From Lemma  \ref{simple lemma}, and the conditions} \Ref{zero 1}, \Ref{zero 2} we obtain that
\be
\frac{p(z_1,z_1,w_1,w_1)}{z_1-\kappa^{-1}w_1}\Bigl|_{w_1=\kappa z_1}
=\frac{p(z_1,z_2,\kappa z_1,\kappa z_2)}{z_2-z_1}\Bigl|_{z_2=z_1}.
\en
Using this identity, it is straightforward to check that
\be
\lefteqn{\frac{q-q^{-1}}{\kappa^2\mdf{z_1^3}}\on{Res}_{w_1=\kappa z_1}\on{Res}_{z_2=z_1}\on{Res}_{w_2=w_1}R(z_1,z_2,w_1,\mdf{w_2})}\\&&
=\mdf{\Big\{}\frac{q-q^{-1}}{\kappa z_2}\on{Res}_{w_2=\kappa z_2}
\frac{q-q^{-1}}{\kappa z_1}\on{Res}_{w_1=\kappa z_1}
\frac{q^{-1}(\kappa z_1-w_2)}{\mdf{\kappa z_1-q_2^{-1}w_2}}R(z_1,z_2,w_1,w_2)\mdf{\Big\}}\Big|_{z_2=z_1}.
\en
Similarly one obtains an equation involving residues at $z_1=\kappa w_1$ and $z_2=\kappa w_2$. 

These two equations are equivalent to the needed relation
\be
\lefteqn{[E_{n-1|0}^\perp (z),F_{n-1|0}^\perp (w)]}\\&&=\frac{1}{q-q^{-1}}(\delta(\kappa^{-1}w/z)K_{n-1}^{+,\perp}(q_1 z)K_0^{+,\perp}(z)-
\delta(\kappa^{-1}z/w)K_{n-1}^{-,\perp}(q_1w)K_0^{-,\perp}(w)).\qed
\en

We denote the subalgebra of $\mathcal E_n$ described in the theorem by $\mathcal E_{n-1}^{n-1|0}$.

Note that Theorem \ref{sub alg} only shows that $\mathcal E_{n-1}^{n-1|0}$ is a factor of toroidal algebra 
$\mathcal E_{n-1}$ with parameters $\tilde q_1,\tilde q_2,\tilde q_3$. However, using the classical limit, see Section \ref{classical}, we obtain that in fact the algebra $\mathcal E_{n-1}^{n-1|0}$ has the same size as $\mathcal E_{n-1}$ and therefore is isomorphic to $\mathcal E_{n-1}(\tilde q_1,\tilde q_2,\tilde q_3)$.
Note that while $\mathcal E_{n-1}^{n-1|0}$ is a subalgebra of a completion of $\mc E_n$, its classical 
limit is a subalgebra of uncompleted classical limit of $\mc E_n$, see Section \ref{classical}.

\medskip

Note also that if $V$ is an admissible representation for $\mc E_n$ then $V$ is an admissible representation of
$\mathcal E_{n-1}^{n-1|0}$. Therefore, 
Theorem \ref{sub alg} justifies the recursive use of the construction of the fused currents,
see Section \ref{fused current}.

\medskip
Let $k\in\{1,\dots, n-1\}$.

Using the theorem recursively, we obtain subalgebras $\mathcal E_k^{k|k+1|\dots|n-1|0}$ generated by currents  $\tilde E_i^\perp (z)=E_i^\perp \bigl(q_1^{\frac{n-k}{k}i}z\bigr)$, 
$\tilde F_i^\perp (z)=F_i^\perp \bigl(q_1^{\frac{n-k}{k}i}z\bigr)$, 
$\tilde K_i^{\pm,\perp}(z)=K_i^{\pm,\perp}\bigl(q_1^{\frac{n-k}{k}i}z\bigr)$, $i=1,\dots,k-1$, and $\tilde E_0^\perp (z)=E_{k|k+1|\dots|n-1|0}^\perp (z)$, $\tilde F_0^\perp (z)=F_{k|k+1|\dots|n-1|0}^\perp (z)$, $\tilde K_0^{\pm,\perp}(z)=K_{k|k+1|\dots|n-1|0}^{\pm,\perp}(z)$. 

The subalgebra $\mathcal E_k^{k|k+1|\dots|n-1|0}$ is isomorphic to 
$\mathcal E_{k}(\tilde q_1,\tilde q_2,\tilde q_3)$ with $\tilde q_i$ given by
\be
\tilde q_1=q_1\cdot q_1^{\frac{n-k}{k}}, \qquad \tilde q_2=q_2, \qquad \tilde q_3=q_3\cdot q_1^{-\frac{n-k}{k}}.
\en

In the same way, we obtain subalgebras $\mathcal E_k^{n-k|\dots|1|0}$ which are generated by currents $\tilde E_i^\perp (z)=E_i^\perp \bigl(q_3^{\frac{n-k}{k}i}z\bigr)$, 
$\tilde F_i^\perp (z)=F_i^\perp \bigl(q_3^{\frac{n-k}{k}i}z\bigr)$, 
$\tilde K_i^{\pm,\perp}(z)=K_i^{\pm,\perp}\bigl(q_3^{\frac{n-k}{k}i}z\bigr)$, $i=n-k+1,\dots,n-1$, and $\tilde E_0^\perp (z)=E_{n-k|\dots|1|0}^\perp (z)$, $\tilde F_0^\perp (z)=F_{n-k|\dots|1|0}^\perp (z)$, $\tilde K_0^{\pm,\perp}(z)=K_{n-k|\dots|1|0}^{\pm,\perp}(z)$. 
The subalgebra $\mathcal E_k^{n-k|\dots|1|0}$ is isomorphic to 
$\mathcal E_{k}(\tilde q_1,\tilde q_2,\tilde q_3)$ with $\tilde q_i$ given by
\be
\tilde q_1=q_1\cdot q_3^{-\frac{n-k}{k}}, \qquad \tilde q_2=q_2, \qquad \tilde q_3=q_3\cdot q_3^{\frac{n-k}{k}}.
\en
\mdf{We abbreviate $\mathcal E_1^{n-1|\dots|1|0}$ to $\mathcal E_1^{\mdf{n-1||0}}$.}

In Section \ref{classical} below we explain that,  
in the classical limit, the embedding of the subalgebra
$\mathcal E_k^{k|k+1|\dots|n-1|0}$ into the completion of  $\mc E_n$  
corresponds to the embedding of submatrices into the upper-left corner. 
Similarly, the embedding $\mathcal E_k^{n-k|\dots|1|0}$ to the completion of $\mc E_n$ corresponds to the embedding of submatrices into the lower-right corner.

\subsection{Computation of $\tilde H_{i,1}$}\label{Hi}
The algebra $\mc E_{n-1}^{n-1|0}$, see Theorem \ref{sub alg}, 
is defined in terms of the perpendicular generators. It is not easy 
to write the standard 
(non-perpendicular) generators of $\mc E_{n-1}^{n-1|0}$ in terms of standard generators of $\mc E_n$ in general. 
In this section we compute such a formula for $\tilde H_{i,1}$. This is used in Section \ref{branching sec}.

\begin{lem}\label{H1 lem}
For $i=0,\dots,n-2$ we have
\begin{align*}
&\tilde{H}_{i,1}=(-q)q_1^{-\frac{i}{n-1}} \lim_{s\to\infty} q_1^{-s}
T_{n-1|0}^s\left(H_{i,1}+\delta_{i,0}q_1^{-1}H_{n-1,1}\right)\,.
\end{align*}
\end{lem}
\begin{proof}
By the definition of the symbol $\perp$
we have \mdf{$\tilde{H}_{i,1}=\tilde{\theta}\bigl(\tilde{H}^{\perp}_{i,1}\bigr)$.}

First assume $n\geq 3$.
Set $\tilde{d}=d q_1^{\frac{1}{n-1}}$. 
Then, for $1\le i\le n-2$, we have from \eqref{thetaH1}
\mdf{
\begin{align*}
\tilde{H}_{i,1}
&=(-\tilde{d})^{-i}
[[\cdots [ [\cdots [\tilde{E}^{\perp}_{0,0},\tilde{E}^\perp_{n-2,0}]_{q^{-1}},
\cdots ,\tilde{E}^\perp_{i+1,0}]_{q^{-1}},\tilde{E}^\perp_{1,0}]_{q^{-1}},
\cdots,\tilde{E}^\perp_{i-1,0}]_{q^{-1}},\tilde{E}^\perp_{i,0}]_{q^{-2}}\,.
\end{align*}
}
Substituting $\tilde{E}^\perp_{j,0}=E^\perp_{j,0}$ ($1\le j\le n-2$) and
\begin{align*}
\tilde{E}^\perp_{0,0}=(-q)\lim_{s\to\infty}q_1^{-s}T_{n-1|0}^s[E^\perp_{0,0},E^\perp_{n-1,0}]_{q^{-1}}\,,
\end{align*}
we find 
\begin{align*}
\tilde{H}_{i,1}= (-\tilde{d})^{-i}(-q)\lim_{s\to\infty}q_1^{-s}T_{n-1|0}^s
\bigl((-d)^iH_{i,1}\bigr)\,
\end{align*}
which gives the desired result. 

Consider the case $i=0$. Using the quadratic relation 
$[E^\perp_{0,0},E^\perp_{1,0}]_{q^{-1}}=-d[E^\perp_{1,1},E^\perp_{0,-1}]_{q^{-1}}$
we rewrite $H_{n-1,1}$ as follows.
\begin{align*}
H_{n-1,1}=(-d)^{-n+2}[[[\cdots[E^\perp_{1,1},E^\perp_{2,0}]_{q^{-1}},\cdots,
E^\perp_{n-2,0}]_{q^{-1}},
E^\perp_{0,-1}]_{q^{-1}},E^\perp_{n-1,0}]_{q^{-2}}\,. 
\end{align*}
Setting $X=[[E^\perp_{1,1},E^\perp_{2,0}]_{q^{-1}},\cdots,E^\perp_{n-2,0}]_{q^{-1}}$ and 
using \eqref{thetaH01} we obtain 
\begin{align*}
&(-d)^{n-1}\bigl(H_{0,1}+q_1^{-1}H_{n-1,1}\bigr) 
=[[X,E^\perp_{n-1,0}]_{q^{-1}},E^\perp_{0,-1}]_{q^{-2}}
-q [[X,E^\perp_{0,-1}]_{q^{-1}},E^\perp_{n-1,0}]_{q^{-2}}
\\
&\qquad\qquad\quad\mdf{=X[E^\perp_{n-1,0},E^\perp_{0,-1}]_q-q^{-2}[E^\perp_{n-1,0},E^\perp_{0,-1}]_{q^{-1}}X
-(1-q^{-2})E^\perp_{0,-1}XE^\perp_{n-1,0}.}
\end{align*}
\mdf{
It follows that
\begin{align*}
\lim_{s\rightarrow\infty}q_1^{-s}T^s_{n-1|0}(H_{0,1}+q_1^{-1}H_{n-1,1})
=(-d)^{-n+1}[X,E^\perp_{n-1|0,-1}]_{q^{-2}}.
\end{align*}
We obtain the statement by noting that
$\tilde{E}^\perp_{1,1}=q_1^{-\frac{1}{n-1}}E^\perp_{1,1}$.}

The case $n=2$ can be checked directly, by noting 
that $\tilde{H}_{0,1}=\tilde{E}^\perp_{0,0}=
E^\perp_{1|0,0}$.
\end{proof}

\subsection{Commuting subalgebras}\label{commute sec}
We show that the constructed "upper left corner" subalgebras  commute with "lower right corner" subalgebras.
\begin{thm}\label{commute} 
For each $k$, the $\mathcal E_n$ subalgebras $\mathcal E_{n-k}^{k|\dots|1|0}$ 
and $\mathcal E_k^{k|k+1|\dots|n-1|0}$ commute.
\end{thm}
\begin{proof} The theorem is proved by the same techniques as Theorem \ref{sub alg}.
 
For example, let us check the commutativity of $E_1^\perp (z)\in \mathcal E_k^{k|k+1|\dots|n-1|0}$ with 
$E_{k|\dots|1|0}^\perp (w)\in\mathcal E_{n-k}^{k|\dots|1|0}$. We consider the correlation function
\be
\lefteqn{\langle E_1^\perp (z)E_k^\perp (q_3^{k}w_k)E_{k-1}^\perp (q_3^{k-1}\mdf{w_{k-1}})
\dots E_0^\perp (w)\rangle}\\&&
=\frac{p(z,w_k,\dots,w_1,w)}{(z-q_2q_3w_1)(z-q_1q_3^2w_2)(z-q_1^{-1}w)(w_k-w_{k-1})\dots(w_1-w)}.
\en
We need to show that the poles at $z=q_2q_3w_1$, $z=q_1q_3^2w_2$ and $z=q_1^{-1}w$ disappear when we multiply
by $(w_k-w_{k-1})\dots(w_1-w)$ and set $w_k=w_{k-1}=\dots=w_1=w$.
But this follows from the wheel conditions for the Laurent polynomial $p(z,w_k,\dots,w_1,w)$.

Let us check the commutativity of $F_1^\perp (z)\in \mathcal E_k^{k|k+1|\dots|n-1|0}$ with 
$E_{k|\dots|1|0}^\perp (w)\in\mathcal E_{n-k}^{k|\dots|1|0}$. 
We consider the correlation function
\be
\lefteqn{\langle F_1^\perp (z)E_k^\perp (q_3^{k}w_k)E_{k-1}^\perp (q_3^{k-1}w_1)\dots E_0^\perp (w)\rangle}\\&&
=\frac{p(z,w_k,\dots,w_1,w)}{(z-\kappa^{-1} q_3w_1)(z-\kappa q_3w_1)(w_k-w_{k-1})\dots(w_1-w)}.
\en
We need to show that the poles at $z=\kappa^{-1}q_3w_1$, $z=\kappa q_3w_1$ disappear when we multiply
by $(w_k-w_{k-1})\dots(w_1-w)$ and set $w_k=w_{k-1}=\dots=w_1=w$. We have
\be
\lefteqn{\on{Res}_{z=q_3\kappa w_1}\langle F_1^\perp (z)E_k^\perp (q_3^{k}w_k)\dots E_2^\perp (q_3^2w_2)E_1^\perp (w_1q_3)E_0^\perp (w)\rangle}\\&&
=-\frac{q_3\kappa w_1^{-1}}{q-q^{-1}}
\langle E_k^\perp (q_3^{k}w_k)\dots E_2^\perp (q_3^2w_2) K_1^{+,\perp}(q_3w_1)E_0^\perp (w)\rangle\\
&&=-q_3^{-1}d^{-1}\frac{q_3\kappa w_1}{q-q^{-1}}\frac{q_1q_3w_1-w}{w_1-w}
\langle E_k^\perp(q_3^{k}w_k)
\dots E_2^\perp (q_3^2w_2) E_0^\perp (w) K_1^{+,\perp}(q_3w_1)\rangle .
\en
In the last expression there is no pole at $w_2=w_1$. It implies that we have the identity
$p(\kappa q_3w_1,w_k,\dots,w_1,w_1,w)=0$ and the pole $z=\kappa q_3w_1$ disappears.

We omit further details.
\end{proof}

\subsection{Classical limit.}\label{classical}
In this subsection we explain  the meaning of the fused currents 
in the classical limit. 

The quantum toroidal $\mathfrak{gl}_n$ algebra $\mathcal{E}_n=\mathcal{E}_n(q_1,q_2,q_3)$
contains two parameters $q,d$. By the classical limit we mean $q\to 1$. 
The algebra $\mathcal{E}_n(q_1,q_2,q_3)$ in the limit is known to have the following 
description.  

Consider the algebra $\mathcal{A}_n(d)=\mathbb{M}_n\otimes\C[Z^{\pm1},D^{\pm1}]$,  
where $\mathbb{M}_n$ stands for the algebra of $n\times n$ matrices,  
and 
$\C[Z^{\pm 1},D^{\pm1}]$ is the algebra generated by symbols $Z,D$ satisfying $DZ=d^{-n}ZD$. 
We regard  $\mathcal{A}_n(d)$ as a Lie algebra by commutators.
Let $\mathcal{L}_n(d)=\mathcal{A}_n(d)\oplus\C c_1\oplus \C c_2$ be its two-dimensional central extension,
where the Lie bracket is given by
\begin{align*}
[M_1\otimes Z^{r_1}D^{s_1}, M_2\otimes Z^{r_2}D^{s_2}] 
&=\mdf{\left(d^{-nr_2s_1}M_1M_2-d^{-nr_1s_2}M_2M_1\right)}\otimes Z^{r_1+r_2}D^{s_1+s_2}
\\
&
+\delta_{r_1+r_2,0}\delta_{s_1+s_2,0}\, d^{-nr_2s_1}\mathrm{tr}\bigl(M_1M_2\bigr)\cdot
(r_1c_1+s_1c_2)\,, 
\end{align*}
for $M_i\in \mathbb{M}_n$, $r_i,s_i\in\Z$, $i=1,2$. 
Let further $\mathcal{L}_n'(d)$ be the Lie subalgebra of $\mathcal{L}_n(d)$ spanned by $c_1,c_2$ and 
elements $\sum M_{r,s}Z^rD^s\in \mathcal{A}_n(d)$ such that $\mathrm{tr}(M_{0,0})=0$.
The classical limit of $\mathcal{E}_n(q_1,q_2,q_3)$ is 
the universal enveloping algebra $U \mathcal{L}_n'(d)$. 

To see this explicitly, set $K_i^\perp=q^{H_{i,0}^\perp }$, $\kappa=q^{-c_1}$,  $c_2=\sum_{i=0}^{n-1}H_{i,0}^\perp$. 
It is then straightforward to check that the limit $q\to 1$ of 
the defining relations 
for the generators $E_{i,k}^\perp ,F_{i,k}^\perp ,H_{i,k}^\perp $ of $\mathcal{E}_n(q_1,q_2,q_3)$
are satisfied by the following elements of $\mathcal{L}_n'(d)$: 
\begin{align}
&\bar{E}_{i,k}^\perp =
\begin{cases}
E_{i,i+1}\otimes Z^k\, d^{-ik}& (1\le i\le n-1)\,, \\
E_{n,1}\otimes D Z^k\,  & (i=0)\,,\\
\end{cases}
\label{Ebar}\\
&\bar{F}_{i,k}^\perp =
\begin{cases}
E_{i+1,i}\otimes Z^k\, d^{-ik}& (1\le i\le n-1)\,, \\
E_{1,n}\otimes  Z^k D^{-1}\, & (i=0)\,,\\
\end{cases}
\nn\\
&\bar{H}_{i,k}^\perp =
\begin{cases}
\bigl(E_{i,i}-E_{i+1,i+1}\bigr)\otimes Z^k\, d^{-ik}& (1\le i\le n-1)\,, \\
\bigl(d^{-nk}E_{n,n}-E_{1,1}\bigr)\otimes  Z^k \,+c_2\delta_{k,0} & (i=0)\,.\\
\end{cases}
\nn
\end{align}
Here $E_{i,j}\in\mathbb{M}_n$ are the matrix units. 
As it is noted in \cite{M99}, \mdf{the} automorphism $\theta\in \mathrm{Aut}\,\mathcal{E}_n$ reduces 
in the classical limit to  
the Lie algebra automorphism $\bar{\theta}\,\in\mathrm{Aut}\,\mathcal{L}_n'(d)$ given by the rule
\begin{align*}
Z\mapsto D,\quad D \mapsto Z^{-1},\quad c_1\mapsto c_2,\quad c_2\mapsto -c_1,
\end{align*}
and $M\mapsto M$ for $M\in\mathbb{M}_n$.

Let us examine the classical limit 
of the fused currents. For simplicity we consider the case 
$n\ge 3$. 
Recall that the Fourier components of the current $E_{n-1|0}^\perp(z)$ are defined to be
\begin{align*}
E_{n-1|0,r}^\perp =\lim_{s\to \infty} q_1^{-r-s}E_{n-1,r+s}^\perp E_{0,-s}^\perp 
=\lim_{s\to \infty} q_1^{-r-s}[E_{n-1,r+s}^\perp ,E_{0,-s}^\perp ]\,,
\end{align*}
where the second equality is due to the meaning of the completion. 
In view of \eqref{Ebar}, the classical limit of this expression is 
\begin{align*}
\bar{E}_{n-1|0,r}^\perp =d^{-r-s}[\bar{E}_{n-1,r+s}^\perp ,\bar{E}_{0,-s}^\perp]=E_{n-1,1}\otimes D Z^r\,.
\end{align*}
This holds true for all $s$, without taking the limit $s\to\infty$ nor 
introducing the completion. 
Similarly the classical limit of $F_{n-1|0,r}^\perp
$ is 
\begin{align*}
\bar{F}_{n-1|0,r}^\perp =d^{s}[\bar{F}_{0,r+s}^\perp ,\bar{F}_{n-1,-s}^\perp ]=E_{1,n-1}\otimes Z^r D^{-1}\,.
\end{align*}
These elements along with the other generators $\bar{E}_{i,r}^\perp d^{-ir/(n-1)}$,  $\bar{F}_{i,r}^\perp d^{-ir/(n-1)}$
for $1\le i\le n-2$ generate a subalgebra of $\mathcal{L}'_{n}(d)$ isomorphic to
$\mathcal{L}'_{n-1}(\tilde{d})$, where $\tilde{d}=d^{n/(n-1)}$
 (note that $D Z= \tilde{d}^{-n+1}Z D$). 
This is nothing but the one induced from the upper left corner embedding of matrix algebras
\begin{align*}
 \mathbb{M}_{n-1}\hookrightarrow \mathbb{M}_{n},\quad M'\mapsto \Bigl(\begin{matrix}
						     M' & 0\\
                                                     0 & 0 \\
						    \end{matrix}\Bigr)\,.
\end{align*}
In a similar manner the 
classical counterparts of $E_{\mdf{n-1||0},r}^\perp $, $F_{\mdf{n-1||0},r}^\perp$
generate a subalgebra $\mathcal{L}'_{1}(d^n)$ commuting with $\mathcal{L}'_{n-1}(\tilde{d})$. 
The former corresponds to the bottom right corner embedding 
\begin{align*}
 \mathbb{M}_{1}\hookrightarrow \mathbb{M}_{n},\quad M''\mapsto \Bigl(\begin{matrix}
						     0 & 0\\
                                                     0 & M'' \\
						    \end{matrix}\Bigr)\,.
\end{align*}
\section{Branching rules}\label{branching sec}
In this section we study the restriction of various $\mathcal E_n$ modules to the subalgebra
$\mc E_{n-1}^{n-1|0}\otimes \mc E_1^{\mdf{n-1||0}}$.


The logic of the computation is the same in all cases, but we start with Fock spaces, and specifically with $n=2$ where the situation is the easiest to describe.
 
\subsection{Fock modules for $\mathcal E_2$}\label{fock branch sec}
In this section we study decompositions of the modules of level $q$ for $\mathcal E_2$. 

 Consider the module $\mathcal F^{(0)}(u)$, see Section \ref{rep sec}. This module has a basis labeled by partitions. In addition, it is convenient to represent this module by the following familiar picture, see Figure 1. 

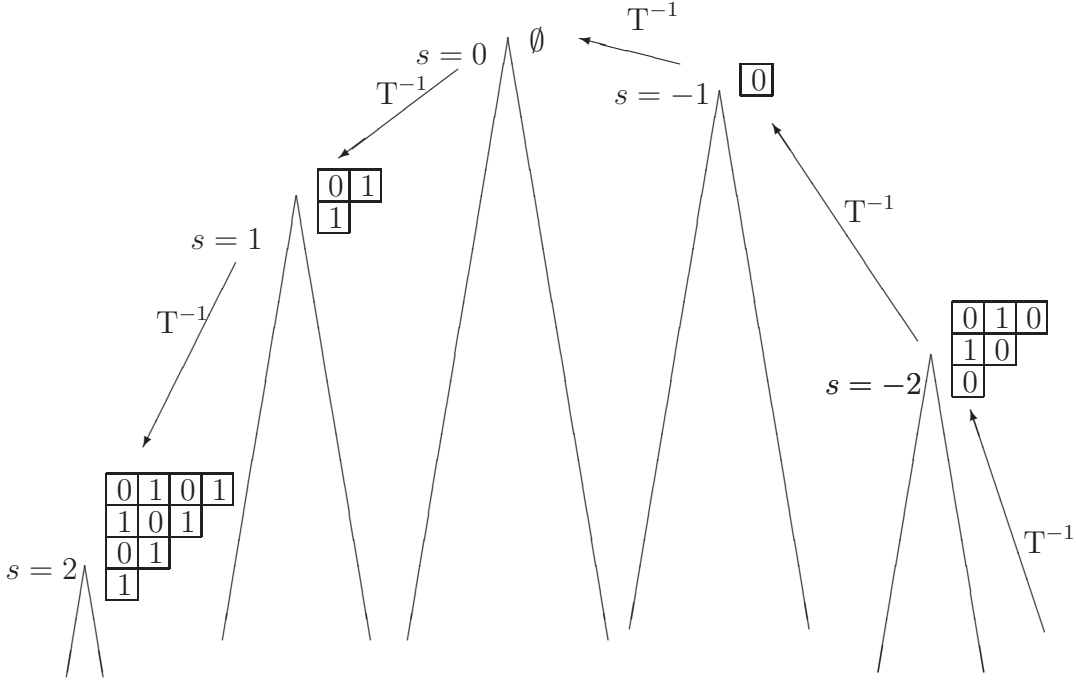
\begin{figure}\label{fock pic}
\begin{picture}(150,300)(20,00)\label{fig}

\put(70,280){\line(1,-6){38}}
\put(70,280){\line(-1,-6){38}}
\put (78,275){$\emptyset$}
\put (35,270){$s=0$}
\put(51,268){\vector(-4,-3){45}}
\put(20,255){$\on{T}^{-1}$}

\put(150,260){\line(1,-6){34}}
\put(150,260){\line(-1,-6){34}}
\put (110,255){$s=-1$}
\put(135,270){\vector(-4,1){38}}
\put(115,283){$\on{T}^{-1}$}
\put(158,270){\line(0,-1){12}}
\put(170,270){\line(0,-1){12}}
\put(158,270){\line(1,0){12}}
\put(158,258){\line(1,0){12}}
\put(162,260){$0$}

\put(-10,220){\line(1,-6){28}}
\put(-10,220){\line(-1,-6){28}}
\put(-50,200){$s=1$}
\put(-33,195){\vector(-1,-2){35}}
\put(-63,168){$\on{T}^{-1}$}
\put(-2,230){\line(0,-1){24}}
\put(10,230){\line(0,-1){24}}
\put(22,230){\line(0,-1){12}}
\put(-2,230){\line(1,0){24}}
\put(-2,218){\line(1,0){24}}
\put(-2,206){\line(1,0){12}}
\put(2,220){$0$}
\put(14,220){$1$}
\put(2,208){$1$}

\put(230,160){\line(1,-6){20}}
\put(230,160){\line(-1,-6){20}}
\put (190,145){$s=-2$}
\put(225,165){\vector(-2,3){55}}
\put(197,210){$\on{T}^{-1}$}

\put(238,180){\line(0,-1){36}}
\put(250,180){\line(0,-1){36}}
\put(262,180){\line(0,-1){24}}
\put(274,180){\line(0,-1){12}}
\put(238,180){\line(1,0){36}}
\put(238,168){\line(1,0){36}}
\put(238,156){\line(1,0){24}}
\put(238,144){\line(1,0){12}}
\put(242,170){$0$}
\put(254,170){$1$}
\put(266,170){$0$}
\put(242,158){$1$}
\put(242,146){$0$}
\put(254,158){$0$}

\put(-90,80){\line(1,-6){7}}
\put(-90,80){\line(-1,-6){7}}
\put (-120,75){$s=2$}
\put(-82,115){\line(0,-1){48}}
\put(-70,115){\line(0,-1){48}}
\put(-58,115){\line(0,-1){36}}
\put(-46,115){\line(0,-1){24}}
\put(-34,115){\line(0,-1){12}}
\put(-82,115){\line(1,0){48}}
\put(-82,103){\line(1,0){48}}
\put(-82,91){\line(1,0){36}}
\put(-82,79){\line(1,0){24}}
\put(-82,67){\line(1,0){12}}
\put(-78,105){$0$}
\put(-66,105){$1$}
\put(-54,105){$0$}
\put(-42,105){$1$}
\put(-78,93){$1$}
\put(-66,93){$0$}
\put(-54,93){$1$}
\put(-78,81){$0$}
\put(-66,81){$1$}
\put(-78,69){$1$}

\put (190,145){$s=-2$}
\put(273,55){\vector(-1,3){28}}
\put(265,85){$\on{T}^{-1}$}
\end{picture}
\caption{The $\mathcal E_2$ module $\mathcal F^{(0)}(u)$}
\end{figure}
On this picture the module $\mathcal F^{(0)}(u)$ looks \mdf{similar to that of} the vacuum
$\widehat{\mathfrak{sl}}_2$ integrable module of \mdf{(additive)}
level one, \mdf{but actually it is not the same.
It is similar} simply because the Fock module restricted to the horizontal algebra
$U_q^{hor}\widehat{\mathfrak{sl}}_2$ is a level \mdf{$q$  module (in the sense  $\kappa^{-1}=q$)}. However, the reader should be warned that our space is in fact the vacuum
$U_q\widehat{\mathfrak{gl}}_2$ module. In other words, we have a Heisenberg current commuting
with the $U_q^{hor}\widehat{\mathfrak{sl}}_2$, see Section \ref{embedding}, 
and our module is the tensor product of the Fock space of the Heisenberg algebra with the vacuum
$U_q^{hor}\widehat{\mathfrak{sl}}_2$ integrable module of level \mdf{$q$}. Thus the module
$\mathcal F^{(0)}(u)$ is the vacuum $U_q^{hor}\widehat{\mathfrak{gl}}_2$ module.

We have the usual $\mathfrak{sl}_2$ weight decomposition given by values of $K_1K_0^{-1}$.
\mdf{We called this weight ``cweight" (see \eqref{weight})}.

The \mdf{cweight} of a partition is given by
$\sharp\{\on{boxes\ of\ color\ 1}\}-\sharp\{\on{boxes\ of\ color\ 0}\}$
in the corresponding colored Young diagram. On Figure 1, the \mdf{cweight} increases from the right to
the left and it is denoted by $s$. The cones which look downward picture vectors of
the same \mdf{cweight}.

We also have the principal degree given by $\on{pdeg} E_i(z)=1$, $\on{pdeg} F_i(z)=-1$. It counts the
total number of boxes and in Figure 1 the principal degree increases from the top to the bottom.

The action of the $i$-th generator of the Heisenberg current increases the \mdf{principal}
degree by $2i$ and does not change the \mdf{cweight}.

\mdf{
We have the action of the shift element $T$, see \Ref{shift element},
on the Fock space as shown in Figure 1. Precisely,
we have  $T^{-1}=s_1s_0$ where the $s_i$ are the Lusztig simple reflections.}

\medskip

Our first observation is the following combinatorial "tensor product" decomposition of the sector with large \mdf{cweight} $s$. Let $\La^s=(2s,2s-1,\dots,2,1)$ and $\La^{-s}=(2s-1,2s-2,\dots,2,1)$
for $s>0$ and let $\La^0$ be the empty partition. Then
$\ket{\La^s}$ is the vector of the lowest degree of \mdf{cweight} $s$.
The degree of $\ket{\La^s}$ is $s(2s+1)$. Fix two partitions $\la$ and $\mu$ with, say,
$k$ parts  and let \mdf{$|s|$} be larger than $k$. Let $\La_{\la,\mu}^s$ be the unique partition
of degree $s(2s+1)+2|\la|+2|\mu|$ and \mdf{cweight} $s$ such that for $i=1,\dots,k$ we have
\be
(\La_{\la,\mu}^s)_i=\La^s_i+2\la_i, \\
(\La_{\la,\mu}^s)_i'=(\La^s)_i'+2\mu_i.
\en
Informally speaking, $\La_{\la,\mu}^s$ is obtained from $\La^s$ by attaching the partitions of $\la$ and $\mu$ made out of dominoes to the top and the bottom respectively, see Figure 2.

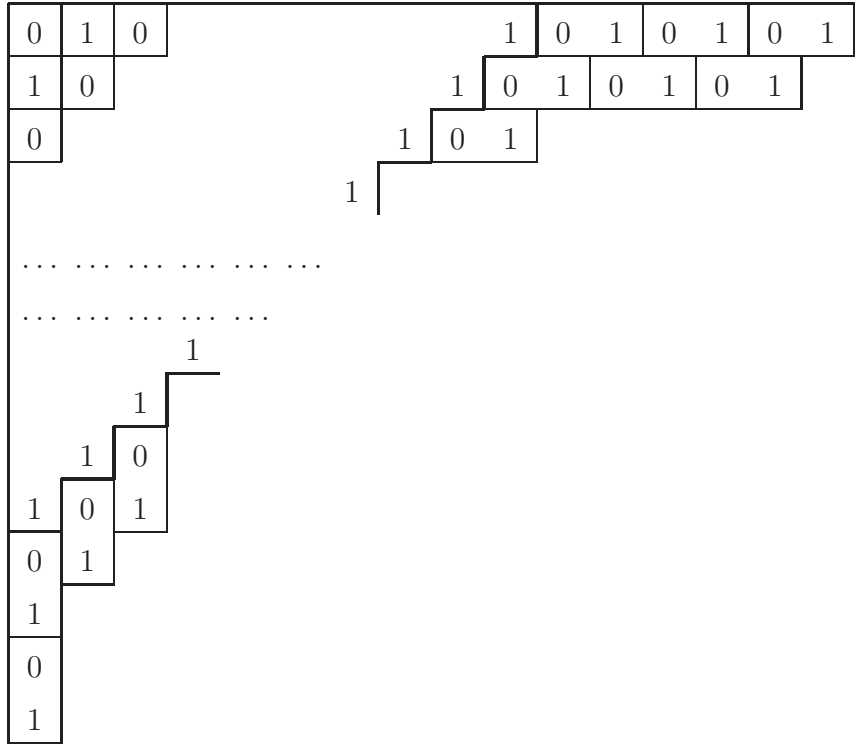
\begin{figure}
\begin{picture}(150,300)(150,00)

\put(270,280){\line(1,0){120}}
\put(270,260){\line(1,0){120}}
\put(250,240){\line(1,0){120}}
\put(230,220){\line(1,0){40}}

\put(270,240){\line(0,-1){20}}

\put(290,260){\line(0,-1){20}}
\put(310,280){\line(0,-1){20}}
\put(330,260){\line(0,-1){20}}
\put(350,280){\line(0,-1){20}}
\put(370,260){\line(0,-1){20}}
\put(390,280){\line(0,-1){20}}

\put(277,265){$0$}
\put(297,265){$1$}
\put(317,265){$0$}
\put(337,265){$1$}
\put(357,265){$0$}
\put(377,265){$1$}

\put(257,245){$0$}
\put(277,245){$1$}
\put(297,245){$0$}
\put(317,245){$1$}
\put(337,245){$0$}
\put(357,245){$1$}

\put(237,225){$0$}
\put(257,225){$1$}

\put(257,265){$1$}
\put(237,245){$1$}
\put(217,225){$1$}
\put(197,205){$1$}

\put(70,80){\line(0,-1){80}}
\put(90,80){\line(0,-1){80}}
\put(110,100){\line(0,-1){40}}
\put(130,120){\line(0,-1){40}}
\put(70,0){\line(1,0){20}}
\put(70,40){\line(1,0){20}}
\put(90,60){\line(1,0){20}}
\put(110,80){\line(1,0){20}}
\put(77,5){$1$}
\put(77,25){$0$}
\put(77,45){$1$}
\put(77,65){$0$}
\put(77,85){$1$}
\put(97,65){$1$}
\put(97,85){$0$}
\put(97,105){$1$}
\put(117,85){$1$}
\put(117,105){$0$}
\put(117,125){$1$}
\put(137,145){$1$}

\put(70,260){\line(1,0){60}}
\put(70,240){\line(1,0){40}}
\put(70,220){\line(1,0){20}}

\put(77,265){$0$}
\put(117,265){$0$}
\put(77,245){$1$}
\put(97,265){$1$}
\put(97,245){$0$}
\put(77,225){$0$}

\put(90,280){\line(0,-1){60}}
\put(110,280){\line(0,-1){40}}
\put(130,280){\line(0,-1){20}}
\thicklines

\put(70,280){\line(1,0){200}}
\put(70,280){\line(0,-1){200}}
\put(270,280){\line(0,-1){20}}
\put(270,260){\line(-1,0){20}}
\put(250,260){\line(0,-1){20}}
\put(250,240){\line(-1,0){20}}
\put(230,240){\line(0,-1){20}}
\put(230,220){\line(-1,0){20}}
\put(210,220){\line(0,-1){20}}
\multiput(75,180)(20,0){6}{$\dots$}
\multiput(75,160)(20,0){5}{$\dots$}

\put(70,80){\line(1,0){20}}
\put(90,80){\line(0,1){20}}

\put(90,100){\line(1,0){20}}
\put(110,100){\line(0,1){20}}
\put(110,120){\line(1,0){20}}
\put(130,120){\line(0,1){20}}
\put(130,140){\line(1,0){20}}
\end{picture}
\caption{The partition $\La_{\la,\mu}^s$ with $\la=(3,3,1)$ and $\mu=(2,1,1)$.}
\end{figure}
Denote by $S^s$ the subspace of $\mathcal F^{(0)}(u)$ of \mdf{cweight} $s$.
Denote by $S^s_{\leq 2k}$ the subspace of $S^s$ consisting of vectors which have degree
at most $s(2s+1)+2k$.
We have the following purely combinatorial lemma.
\begin{lem}\label{la mu}
If $2s>k$, the vectors $\ket{\La^s_{\mdf{\la,\mu}}}$ with $|\mu|+|\la|\leq k$ form a basis of $S^s_{\leq 2k}$.
\end{lem}
\begin{proof} 
Partitions of vectors in $S^s_{\leq 2k}$ have $s$ more boxes of color $1$ than color $0$. Each odd row contains at least as many boxes of color $0$ as color $1$. Each even row contains at most one more box of color $1$ than color $0$. It follows that every such partition contains $\La^s$. Hence any such partition is obtained from $\La^s$  by adding $r$ boxes of color $1$ and $r$ boxes of color $0$ where $r\leq k$. It is easy to see that the only way to do it is as described in the lemma.
\end{proof}

Another important statement is the following lemma. 

\begin{lem}\label{stab lem}
Let $\la,\mu$ be partitions and let \mdf{$2s> |\la|+|\mu|$}. Then $T^{-1} \ket{\La^s_{\mdf{\la,\mu}}}=a_{s,\la,\mu}\ket{\La^{s+1}_{\mdf{\la,\mu}}}$, where $a_{s,\la,\mu}$ is a non-zero constant. 
\end{lem}
\begin{proof} 
For $s\geq 0$, one cannot remove boxes of color $0$ from 
$\ket{\La^s_{\mdf{\la,\mu}}}$. It is therefore 
a lowest weight vector with respect to $U_q\mathfrak{sl}_2$
generated by $E_{0,0},F_{0,0}$. Then 
$s_0\ket{\La^s_{\mdf{\la,\mu}}}$ is a non-zero constant multiple of the corresponding highest weight vector, that is 
$\ket{\La^{-s-1}_{\mdf{\la,\mu}}}$. Similarly 
$s_1\ket{\La^{\mdf{-s-1}}_{\mdf{\la,\mu}}}$ is a non-zero constant multiple of
$\ket{\La^{\mdf{s+1}}_{\mdf{\la,\mu}}}$. 
The lemma follows.
\end{proof}
{
For $s\in\Z$, and partitions $\la,\mu$, we choose any integer $r=r(s,\la,\mu)$ such that $2(r+s)> |\la|+|\mu|$ and 
define vectors $v^s_{\la,\mu}$ by the formula
\be
v^s_{\la,\mu}= T^{r} \ket{\La^{r+s}_{\la,\mu}}.
\en
Note that different choices of $r$ change vectors $v^s_{\la,\mu}$ by non-zero scalars.

\begin{cor}\label{stable basis}
The vectors $v^s_{\la,\mu}$ form a basis of the space $S^s$.
\end{cor}
\begin{proof}
The corollary follows from Lemma \ref{stab lem} and Lemma \ref{la mu}.
\end{proof}

Recall that we have the subalgebra $\mathcal E_1^{(1)}\otimes \mathcal E_1^{(3)}\subset \mathcal E_2$.
The subalgebra $\mathcal E_1^{(1)}$ is generated by currents $E_{1|0}^{(1),\perp}(z)$, $F_{1|0}^{(1),\perp}(z)$, $K_{1|0}^{\pm(1),\perp}(z)$ and 
the subalgebra $\mathcal E_1^{(3)}$ is generated by currents $E_{1|0}^{(3),\perp}(z)$, $F_{1|0}^{(3),\perp}(z)$, $K_{1|0}^{\pm(3),\perp}(z)$. Now we are in a position to establish the decomposition of the $\mathcal E_2$-module $\mc F^{(0)}$. 

We write $\mc F^{(2;0)}(u)$ for the $\mathcal E_2$ Fock module $\mc F^{(0)}(u)$, we also write $\mc F^{(1)}(u)$ (resp.
$\mc F^{(3)}(u)$) for the $\mathcal E_1^{(1)}$ (resp. $\mathcal E_1^{(3)}$) Fock modules.

\begin{thm}\label{Fock 2 thm}
We have an isomorphism of $\mathcal E_1^{(1)}\otimes \mathcal E_1^{(3)}$ modules
\bea\label{Fock 2 decomp}
\mathcal F^{(2;0)}(u)=\mathop{\oplus}\limits_{s\in\Z}\ x^{s(2s+1)}z^s\ \mathcal F^{(1)}(-qq_1^{2s}u)\boxtimes \mathcal F^{(3)}(-qq_3^{2s}u).
\ena
In particular, this isomorphism identifies the space $\mc F^{(1)}(-qq_1^{2s}u)\boxtimes \mathcal F^{(3)}(-qq_3^{2s}u)$ with the subspace of $\mathcal F^{(2;0)}(u)$ of \mdf{cweight} $s$.

Here the factor $x^{s(2s+1)}z^s$ signifies the \mdf{cweight} and degree of the top vector of
the subspace $\mc F^{(1)}(-qq_1^{2s}u)\boxtimes \mathcal F^{(3)}(-qq_3^{2s}u)$ in $\mc F^{(2;0)}(u)$. 
\end{thm}
\begin{proof}
The algebras $\mc E^{(1)}_1$ and $\mc E^{(3)}_1$ are defined in terms of the perpendicular generators of $\mc E_2$  and the action of $\mc E_2$ in $\mathcal F^{(2;0)}(u)$ is given in terms of usual generators. Therefore, in general, it is not easy to compute the action of algebras $\mc E^{(1)}_1$ and $\mc E^{(3)}_1$ in  $\mathcal F^{(2;0)}(u)$.
However, at least we have the following formula, see Lemma \ref{H1 lem}, which turns out to be sufficient for our purposes.
\bea
H_{0,1}^{(1)}=(-q)\lim_{s\to \infty} q_1^{-2s}T^s(q_1^{-1}H_{1,1}+H_{0,1}),\label{H011}\\
H_{0,1}^{(3)}=(-q)\lim_{s\to \infty} q_3^{-2s}T^s(q_3^{-1}H_{1,1}+H_{0,1}).\label{H012}
\ena
It follows that we can compute 
\bea\label{comp H01}
\bra{v}H_{0,1}^{(1)}\ket{v}=\mdf{-q}
\lim_{s\to \infty} \bra{\mdf{T^{-s}}v}q_1^{-2s}(q_1^{-1}H_{1,1}+H_{0,1})\ket{\mdf{T^{-s}}v},\\
\bra{v}H_{0,1}^{(3)}\ket{v}=\mdf{-q}
\lim_{s\to \infty} \bra{\mdf{T^{-s}}v}q_3^{-2s}(q_3^{-1}H_{1,1}+H_{0,1})\ket{\mdf{T^{-s}}v}.\nn
\ena
\newpage
We use these formulas to establish
\bea\label{hw}
H_{0,1}^{(1)}\ket{\La^s}=q^2uq_1^{2s}\ket{\La^s},\qquad
H_{0,1}^{(3)}\ket{\La^s}=q^2uq_3^{2s}\ket{\La^s}.
\ena
Now note, that for all $s\in \Z$ the vectors $\ket{\La^s}$ are lowest weight vectors with respect to 
$\mc E^{(1)}_1$ and $\mc E^{(3)}_1$ for the degree reasons. Since the levels of $\mc E^{(1)}_1$ and $\mc E^{(3)}_1$
coincide with the level of $\mc E_2$, the vector $\ket{\La^s}$ generates a level $q$ module for both of these algebras. Such a module necessarily contains a Fock module. The evaluation parameter of the Fock module is now obtained from \Ref{hw}. 

It follows that $\mc F^{(2;0)}(u)$ contains the right hand side of \Ref{Fock 2 decomp}. Then the equality follows from Lemma \ref{stable basis}.
\end{proof}
\begin{rem} {\rm
Using \Ref{comp H01}, we can compute the action of operators $H_{0,1}^{(1)}$ and $H_{0,1}^{(3)}$ on basis $v^s_{\la,\mu}$. Moreover, in fact, the spectrum of the operator $H_{0,1}$ is simple in the $\mc E_1$ Fock module. It allows us to identify (up to a constant) the vector $v^s_{\la,\mu}$ with the vector $\ket{\la}\boxtimes\ket{\mu}$.}
\end{rem}

\subsection{Fock modules for $\mathcal E_n$}\label{fock n branch sec}
In this section we generalize the results of Section \ref{fock branch sec} for all Fock modules.

\mdf{Fix $p\in\{0,\dots,n-1\}$, and
consider the $\mathcal E_n$ Fock module $\mathcal F^{(p)}(u)$.}
Then we have a picture, similar to Figure 1, where the lattice of roots is now $\Z_{n-1}$.
The top vectors (or extremal vectors) are obtained by the action of the braid group
on the $\ket{\emptyset}$.

Denote the simple roots of $\mathfrak{sl}_n$ 
by $\alpha_j$, $j=1,\cdots,n-1$.
Let $\eta,\ \eta^{(p)},\ \eta_i$ be the following $\mathfrak{sl}_n$ roots:
\be
\eta=\sum_{j=1}^{n-1}j\al_j,\qquad \eta^{(p)}=\sum_{j=1}^{p-1}j\al_j+p\sum_{j=p}^{n-1}\al_j,\qquad
\eta_i=\sum_{j=1}^{i}(i-j+1)\al_{n-j}.
\en
Here $i=0,\dots,n-2$.

Given an $\mathfrak{sl}_n$ root $\gamma$, there is unique $s,a_1,\dots,a_{n-2}\in\Z$, $i\in\{0,1,\dots,n-2\}$,
such that
\bea\label{all weights}
\gamma=\eta^{(p)}+\eta_i+s\eta+\sum_{j=1}^{n-2}a_j\al_j.
\ena

For $s\equiv i+p\ (\on{mod} \ n-1)$ denote $v^{s,i}$ the extremal vector of \mdf{cweight} 
\be
w(s,i,p):=\eta^{(p)}+\eta_i+\frac{s-i-p}{n-1}\eta.
\en
\begin{lem} For $ns\geq i+p$, we have $v^{s,i}=\ket{\La^{s,i}}$, where the partition
$\La^{s,i}$ has $ns-p$ non-trivial parts given by
\be
(\La^{s,i})_j=\frac{\mdf{ns-i-p}}{n-1}-\left[\frac{j-i-1}{n-1}\right],
\en
see Figure 3. 
\end{lem}
\begin{proof}
The principal degree (or total number of boxes) of $\La^{s,i}$ is easily computed and is given by
\be
\rho(s,i,p)=\frac{(ns+n-i-p-1)(ns+i-p)}{2(n-1)}.
\en
Then one checks by a straightforward computation that
\be
\rho(s,i,p)=\frac{n}{2}
\bigl( (w(s,i,p)-\bs \La_p,w(s,i,p)-\bs \La_p)-(\bs \La_p,\bs \La_p)\bigr)+
(w(s,i,p),\sum_{l=0}^{n-1}{\bs\La}_l)
\en
Here $\bs\La_p$ denotes the $p$-th fundamental $\widehat{\mathfrak{sl}}_n$ weight. The lemma follows.
\end{proof}

\mdf{
\begin{lem}\label{lowest degree}
The vector $\ket{\La^{s,i}}$ has the smallest principal degree among
all vectors of \mdf{cweights} \Ref{all weights} with the given $s,i$ and various $a_i$.
\end{lem}
\begin{proof}
Note that the Fock module $\F^{(p)}(u)$ is irreducible, and
 $F_{i,k}\ket{\La^{s,i}}=0$ for $i=0,\dots,n-2$. The statement of the lemma follows.
\end{proof}
}
\medskip

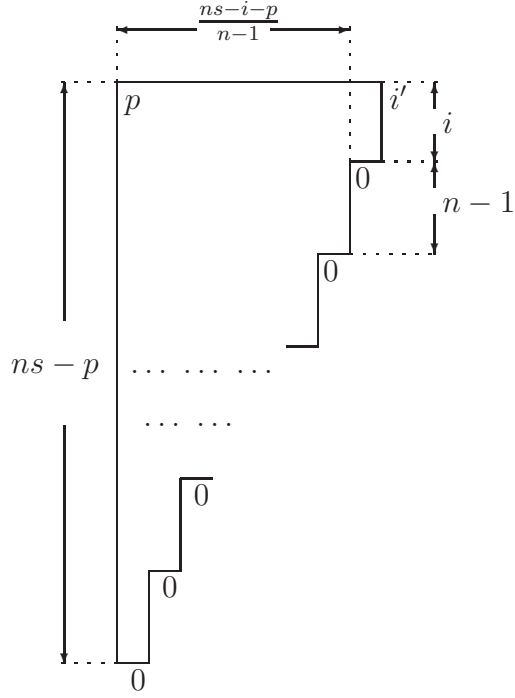
\begin{figure}\label{fock n pic}
\begin{picture}(150,270)(20, 30)

\put(108,280){\vector(1,0){30}}
\put(80,280){\vector(-1,0){30}}
\put(80,280){$\frac{ns-i-p}{n-1}$}

\put(30,170){\vector(0,1){90}}
\put(30,130){\vector(0,-1){90}}
\put(10,150){$ns-p$}

\put(170,250){\vector(0,1){10}}
\put(170,240){\vector(0,-1){10}}
\put(173,241){$i$}

\put(170,220){\vector(0,1){10}}
\put(170,210){\vector(0,-1){15}}
\put(173,211){$n-1$}

\put(50,260){\line(1,0){100}}
\put(50,260){\line(0,-1){220}}
\put(150,260){\line(0,-1){30}}
\put(150,230){\line(-1,0){12}}
\put(138,230){\line(0,-1){35}}
\put(138,195){\line(-1,0){12}}
\put(126,195){\line(0,-1){35}}
\put(126,160){\line(-1,0){12}}

\multiput(55,150)(20,0){3}{$\dots$}
\multiput(60,130)(20,0){2}{$\dots$}

\put(50,40){\line(1,0){12}}
\put(62,40){\line(0,1){35}}
\put(62,75){\line(1,0){12}}
\put(74,75){\line(0,1){35}}
\put(74,110){\line(1,0){12}}

\put(53,250){$p$}
\put(140,220){$0$}
\put(128,185){$0$}
\put(55,30){$0$}
\put(67,65){$0$}
\put(79,100){$0$}

\put(153,250){$i'$}
\multiput(138,220)(0,5){13}{\line(0,1){1}}
\multiput(150,260)(5,0){5}{\line(1,0){1}}
\multiput(150,230)(5,0){5}{\line(1,0){1}}
\multiput(138,195)(5,0){7}{\line(1,0){1}}

\multiput(50,260)(0,5){5}{\line(0,1){1}}
\multiput(50,260)(-5,0){5}{\line(-1,0){1}}
\multiput(50,40)(-5,0){5}{\line(-1,0){1}}

\end{picture}
\caption{The staircase partition $\La^{s,i}$. Here the color $i'$ is $n-i-1$.}
\end{figure}
Let $T$ be the automorphism of $\mc E_n$ given by $T=T_{n-1|0}^n$, see \Ref{T}. After restriction to the horizontal subalgebra $U^{hor}_q\bigl(\widehat{\mathfrak{sl}}_n\bigr)$, $T$ becomes the translation operator in the braid group. in terms of the Lusztig simple reflections we have, see Lemma \ref{Weyl-translation},
\be
T^{-1}=(s_{n-1}\dots s_2 s_1 s_0)^{n-1}.
\en
Note also that $T$ act as identity on the Heisenberg algebra $\mathfrak{a}^{hor}$.
Hence the operator $T^{-1}$ acts on $\mathcal F^{(p)}(u)$ and changes \mdf{cweight}
by $\eta$. In particular, we have
\be
T^{-1} v^{s,i}=a_{n,s,i} v^{s+n-1,i},
\en
for some non-zero constants $a_{n,s,i}$.

We recall that we have the subalgebra $\mathcal E_{n-1}^{\mdf{n-1|0}}\otimes \mathcal E_1^{\mdf{n-1||0}}\subset \mathcal E_n$.

The subalgebra $\mathcal E_{n-1}^{n-1|0}$ is generated by currents $E_{n-1|0}^\perp(z)$, $F_{n-1|0}^\perp(z)$,$K_{n-1|0}^{\pm,\perp}(z)$ and  $E_i^\perp\bigl(q_1^{\frac{i}{n-1}}z\bigr)$, $F_i^\perp\bigl(q_1^{\frac{i}{n-1}}z\bigr)$, $K_i^{\pm,\perp}\bigl(q_1^{\frac{i}{n-1}}z\bigr)$ with $i=1,\dots, n-2$.

The subalgebra $\mathcal E_1^{\mdf{n-1||0}}$ is generated by currents
\mdf{
$E^\perp_{n-1||0}(z)$,$F^\perp_{n-1||0}(z)$,$K^{\pm,\perp}_{n-1||0}(z)$,
}
see \Ref{n-0}.

By Theorem \ref{commute} the subalgebras $\mathcal E_{n-1}^{n-1|0}$ and $\mathcal E_1^{\mdf{n-1||0}}$ commute inside $\mc E_n$.

\mdf{
The following lemma \textcolor{jimbo}{follows} 
from the construction of the subalgebras and Lemma \ref{perp deg}.
\begin{lem}\label{graded embedding}
Let $\deg^{(n-1)}$ and $\deg^{(n)}$ denote the degree in $\E_{n-1}$ and $\E_n$, respectively.
The embedding of $\E_{n-1}$ is graded. Namely, if $x\in\E_{n-1}$ is a graded element such that
$\deg^{(n-1)}x=(\ell,\ell+d_1,\ldots,\ell+d_{n-2},k)$, then the embedded element,
which we denote also by $x$, is graded and $\deg^{(n)}x=(\ell,\ell+d_1,\ldots,\ell+d_{n-2},\ell,k)$.
Similarly, if $x\in\E_1$ is such that $\deg^{(1)}=(\ell,k)$ then $\deg^{(n)}=(\ell,\ell,\ldots,\ell,k)$.
\end{lem}
}

Now we are in a position to describe the decomposition of the $\mathcal E_n$-module $\mc F^{(p)}(u)$. 

We write $\mc F^{(n;p)}(u)$ for the $\mathcal E_n$ Fock module, similarly we write $\mc F^{(n-1;k)}(u)$ (resp.
$\mc F^{(1)}(u)$) for the $\mathcal E_{n-1}^{n-1|0}$ (resp. $\mathcal E_1^{\mdf{n-1||0}}$) Fock modules.

\begin{thm}\label{Fock n thm}
We have an isomorphism of $\mathcal E_{n-1}^{\mdf{n-1|0}}\otimes \mathcal E_1^{\mdf{n-1||0}}$ modules
\be
\mathcal F^{(n;p)}(u)=\mathop{\oplus}\limits_{i=0}^{n-2}\mathop{\oplus}\limits_{\buildrel {s\in\Z, }\over {s\equiv p+i (\on{mod}n-1)}}
\ x^{\rho(s,i,p)}z^{w(s,i,p)}\ \mathcal F^{(n-1;n-i-1)}(-qq_1^{\frac{ns-p}{n-1}}u)\boxtimes \mathcal F^{(1)}(-qq_3^{ns-p}u).\hspace{-25pt}
\en
\end{thm}
\begin{proof}
The theorem is proved similarly to Theorem \ref{Fock 2 thm}.
\mdf{
The first step is to
show that the vector $\La^{s,i}$ is the lowest weight vectors for both of the actions of
$\E_{n-1}$ and $\E_1$. In the proof we use Lemma \ref{lowest degree} and Lemma
\ref{graded embedding}.
}

Let us indicate the combinatorial picture. Similarly to $n=2$ case, we define the basis $\ket{\La^{s,i}_{\la,\mu}}$ and the partition $\La^{s,i}_{\la,\mu}$ is obtained
from $\La^{s,i}$ by adding legs in the shape of $\mu$, and arms in the shape of $\la$. Moreover, each box of 
$\mu$ is replaced with vertical strip of $n$ boxes colored $0,1,2,\dots, n-1$ (from top to bottom). The partition $\la$ is colored by $n-1$ colors, with the top left box being $n-i-1$. Then $0$ boxes of $\la$ are replaced in $\La^{s,i}_{\la,\mu}$ by horizontal dominoes colored $0,n-1$ from left to right. The other boxes of $\la$ go to the boxes of the same color in $\La^{s,i}_{\la,\mu}$.

Then we use Lemma \ref{H1 lem} in place of \Ref{H011}, \mdf{\Ref{H012}}.

We leave the rest of the details to the reader.
\end{proof}

\subsection{Generic tensor products of Fock modules}\label{products of focks branch sec}
Our next goal is to establish the decomposition of the module $\mc F^{(n,p_1)}(u_1)\otimes \mc F^{(n,p_2)}(u_2)\otimes\dots\otimes \mc F^{(n,p_k)}(u_k)$ with generic evaluation parameters $u_1,\dots,u_k$ as
$\mathcal E_{n-1}^{\mdf{n-1|0}}\otimes \mathcal E_1^{\mdf{n-1||0}}$ module.

We prepare the following lemma.

\begin{lem}\label{h1 determines tensor}
Let $W=\mc F^{(p_1)}(u_1)\otimes \mc F^{(p_2)}(u_2)\otimes\dots\otimes \mc F^{(p_k)}(u_k)$ be a tensor product of $\mc E_n$-Fock modules. Assume that $u_1,\dots,u_k,q_1,q_2$ are algebraically independent over $\Q$, in particular, that $W$ is irreducible.

Let $V$ be an $\mc E_n$ module such that 

\begin{itemize}
\item V is an irreducible lowest weight $\mc E_n$-module $V$ with lowest weight vector $v$.

\item $W$ and $V$ have the same graded character in the principal gradation.
Let $w_1,\dots,w_k$ be a basis of \mdf{the} subspace of $V$ of vectors of \mdf{the principal}
degree one. \mdf{We choose a basis}
consisting of eigenvectors of operators $K_i$ and $H_{i,1}$, $i=0,\dots,n-1$.

\item The eigenvalues of $K_i$, $H_{i,1}$, $i=0,\dots,n-1$, on vectors $v,w_1,\dots,w_k$ coincide with the corresponding eigenvalues of vectors $\ket{\emptyset}^{\otimes k}$, $\ket{\emptyset}^{\otimes {i-1}}\otimes \ket{\{1\}}\otimes \ket{\emptyset}^{\otimes {k-i}}$, $i=1,\dots,k$,
 in $W$.  
 \end{itemize}

Then $\mc E_n$-modules $V$ and $W$ are isomorphic. 
\end{lem}
\begin{proof}
Let $K_iv=q^{-k_i}v$. Then $k_i$ is the number of Fock representations $\F^{(p_j)}(u_j)$ in $W$ such that $p_j=i$. 
Then the number of $w_j$ such that $K_iw_j=q^{-k_i+2}w_j$ is $k_i$. It follows that the $i$-th component of the lowest weight of $v$ has at most $k_i$ zeros and at most $k_i$ poles. 

Therefore the lowest weight of $V$ coincides with that of a product of vacuum Macmahon modules  
$\mc M^{(p_1)}(\tilde u_1, \kappa_1)\otimes \mc M^{(p_2)}(\tilde u_2, \kappa_2)\otimes\dots\otimes \mc M^{(p_k)}(\tilde u_k,\kappa_k)$, see \cite{FJMM2}, for some levels $\kappa_i$ and some evaluation points parameters $\tilde u_j$. 
On the other hand, for the Fock module $\mc F^{(i)}(u)$, 
the difference of 
eigenvalues of $H_{i,1}$ on $|\emptyset\rangle$ and on $|(1)\rangle$ 
is given by $-(q+q^{-1})u$. 
The same holds also for the Macmahon module $\mc M^{(i)}(u,\kappa)$. 
From the hypothesis of the lemma, we then conclude 
that $\{\tilde u_i\}_{i=1}^k=\{u_i\}_{i=1}^k$.  In particular, the tensor product of Macmahon modules is well-defined.

Note that given lowest weight of $V$, the choice of the Macmahon modules (the choice of $\kappa_j$) is not unique, it is defined by the choice of pairing up the factors in the numerator with factors in the denominator. Since zeros in the denominator are algebraically independent, for each factor in the numerator, there is at most one zero in the denominator, such that the corresponding Macmahon module is not irreducible. Let us choose the pairing such that as many Macmahon modules as possible are not irreducible. We can also choose the order of factors in such a way that the tensor product is cyclic.

Then in the tensor product we need to have $k$ linearly independent singular vectors of principal degree $2$.
One can see that it is possible only if $\kappa_i=q$ for all $i$. The lemma follows.
\end{proof}

Note that the subalgebras $\mathcal E_{n-1}^{\mdf{n-1|0}}$ and $\mathcal E_1^{\mdf{n-1||0}}$ are not Hopf subalgebras. However, the decomposition formula looks as a tensor product formula.

\begin{thm}\label{k Fock n thm}
We have an isomorphism of $\mathcal E_{n-1}^{\mdf{n-1|0}}\otimes \mathcal E_1^{\mdf{n-1||0}}$ modules
\bea\label{k Fock n decomp}
\lefteqn{\mathop{\otimes}\limits_{j=1}^k \mc F^{(n,p_j)}(u_j)
=\mathop{\oplus}\limits_{i_1,\dots,i_k=0}^{n-2}\mathop{\oplus}\limits_{\buildrel {s_1,\dots,s_k\in\Z, }\over {s_j\equiv p_j+i_j (\on{mod}n-1)}}
x^{\sum\limits_{j=1}^k\rho(s_j,i_j,p_j)}z^{\sum\limits_{j=1}^kw(s_j,i_j,p_j)}}\notag\\ &&
\hspace{90pt}\times \left(\mathop{\otimes}\limits_{j=1}^k\mathcal F^{(n-1;n-i_j-1)}(-qq_1^{\frac{ns_j-p_j}{n-1}}u_j)\right)\boxtimes \left(\mathop{\otimes}\limits_{j=1}^k\mathcal F^{(1)}(-qq_3^{ns_j-p_j}u_j)\right).
\ena
\end{thm}
\begin{proof}
Clearly it is enough to show that the vectors $\otimes_{j=1}^k \ket{\La^{s_j,i_j}}$ are lowest weight vectors ?with respect to $\mathcal E_{n-1}^{\mdf{n-1|0}}$ and  $\mathcal E_1^{\mdf{n-1||0}}$ and to compute their lowest weights in accordance with \Ref{k Fock n decomp}.

Recall that we defined a basis of the Fock space  $\ket{\La^{s,i}_{\mdf{\la,\mu}}}$, see proofs of
Theorems \ref{Fock 2 decomp} and \ref{Fock n thm}. Consider the basis of
$\mathop{\otimes}\limits_{j=1}^k \mc F^{(n,p_j)}(u_j)$ given by
$\mathop{\otimes}\limits_{j=1}^k\ket{\La^{s_j,i_j}_{\mdf{\la_j,\mu_j}}}$.
Here $i_j=0,\dots,n-2$, $s_j\in\Z$ and $\la_j,\mu_j$ are arbitrary partitions. 

Consider $\mathcal E_{n-1}^{\mdf{n-1|0}}$. By definition, any given
$\tilde g\in \mc E_{n-1}^{\mdf{n-1|0}}$ acts as a limit of operators $q_{1}^{rs} T^s g$ for some $r$
and some $g\in\mc E_n$. We have 
$(T^s g) v= T^s \circ g \circ T^{-s}v$. Also, note that action  of $T$ in the tensor product of modules
is given by the tensor product of action of $T$ in factors.  

Therefore to compute action of $\tilde g$ we have to apply $g$ to an element shifted far to
the stable zone, that is to $\otimes_{j=1}^k \ket{\La^{s_j,i_j}_{\mdf{\la_j,\mu_j}}}$ with large $s_j$. 

If all $s_j$ are large the vector $\otimes_{j=1}^k \ket{\La^{s_j,i_j}_{\mdf{\la_j,\mu_j}}}$
corresponds to a vector given by a product of partitions of the type shown on Figure 2.
Then the operator $g$
acting on this vector \mdf{produces}
a linear combination of vectors corresponding to the product of partitions
where a fixed total amount of boxes has been removed and added.

For any fixed $r$, there exist some $M>0$ such that if all $s_j>M$, there is no vector of degree less then $\otimes_{j=1}^k \ket{\La^{s_j,i_j}}$ which corresponds to a set of partitions that can be obtained from $\La^{s_j,i_j}$ by a change of $r$ boxes. It follows that this vector is a lowest weight vector with respect to $\mathcal E_{n-1}^{\mdf{n-1|0}}$ and similarly with respect to $\mathcal E_1^{\mdf{n-1||0}}$.

Moreover, for the same reason, the span of all vectors
$\otimes_{j=1}^k \ket{\La^{s_j,i_j}_{\mdf{\la_j,\mu_j}}}$ with fixed $s_j,i_j$ and all tuples of partitions
$\la_j,\mu_j$ is stable under the action of the $\mathcal E_{n-1}^{\mdf{n-1|0}}$.

To determine the lowest weight of $\otimes_{j=1}^k \ket{\La^{s_j,i_j}}$, we check the conditions of Lemma \ref{h1 determines tensor}. By a similar argument as above, the vectors of principal degree $1$ are obtained from $\otimes_{j=1}^k \ket{\La^{s_j,i_j}}$ by adding one box or a domino. The action of $H_{i,1}$ is computed then as before.
\end{proof}

\subsection{The modules $\mc N_{\al,\beta}^{(p)}$}\label{m branch sec}
In this section we deduce the decomposition of modules $\mc N_{\al,\beta}^{(p)}(u)$ from Theorem \ref{k Fock n thm}.

Recall that $\mc N_{\al,\beta}^{(p)}(u)$ is a submodule in the tensor product
$\mathcal{F}^{(p_1)}(u_1)\otimes \cdots \otimes \mathcal{F}^{(p_k)}(u_k)$, see \Ref{N in otimes}, where $p_i,u_i$ are given by \Ref{colors}. Recall also, that $a_i,b_i$ are given by \Ref{ab}.

Assume that $p_i\in\{0,\dots, n-1\}$, and define the numbers $m_i$ by
\bea\label{m}
p_i=p_{i+1}+b_i-a_i-m_in.
\ena

\begin{thm}\label{N thm}
We have an isomorphism of $\mathcal E_{n-1}^{\mdf{n-1|0}}\otimes \mathcal E_1^{\mdf{n-1||0}}$ modules
\bea\label{N decomp}
\lefteqn{\mc N_{\al,\beta}^{(n;p)}(u)=\mathop{\oplus}\limits_{i_1,\dots,i_k=0}^{n-2}\mathop{\oplus}\limits_{\buildrel {s_1,\dots,s_k\in\Z, }\over {s_j\equiv p_j+i_j (\on{mod}n-1)}}
x^{\sum\limits_{j=1}^k\rho(s_j,i_j,p_j)}z^{\sum\limits_{j=1}^kw(s_j,i_j,p_j)}}\notag\\ &&
\hspace{190pt}\times \mc N_{\gamma(s),\beta}^{(n-1;n-1-i_k)}(-qq_1^{\frac{ns_1-p_1}{n-1}}u)\boxtimes \mc N_{\gamma(s),\alpha}^{(1)}(-qq_3^{ns_1-p_1}u),
\ena
where
\be
l_j(s)=\mdf{s_j-s_{j+1}+m_j}, \qquad \ga_j(s)-\ga_{j+1}(s)=l_j(s),
\en
and the summation is over $s_1,\dots,s_k$ such that $l_j(s)\geq 0$, $j=1,\dots,k-1$.
\end{thm}
\begin{proof}
Theorem \ref{N thm} is deduced from Theorem \ref{k Fock n thm}. The module 
$\mc N_{\al,\beta}^{(n;p)}(u)$ is the submodule of a tensor product of Fock modules, see \Ref{N in otimes}, 
which is described by conditions \Ref{part cond}.

Therefore, we start from the generic tensor product of Fock modules as in Theorem \ref{k Fock n thm}. Note that the action of all operators depends on evaluation parameters algebraically. Therefore we can specialize the evaluation parameters to any values where the tensor product is well-defined. Let us specialize the evaluation parameters as in \Ref{colors}. Then we discard the representations of
$\mathcal E_{n-1}^{\mdf{n-1|0}}\otimes \mathcal E_1^{\mdf{n-1||0}}$ whose lowest weight vectors
do not satisfy \Ref{part cond}. Next, we check that the surviving lowest weight vectors
have exactly lowest weights of
$\mc N_{\gamma(s),\beta}^{(n-1;n-1-i_k)}(-qq_1^{\frac{ns_1-p_1}{n-1}}u)\boxtimes
\mc N_{\gamma(s),\alpha}^{(1)}(-qq_3^{ns_1-p_1}u)$. It shows that the left hand side
of \Ref{N decomp} contains the right hand side. It remains to see that both sides coincide which is
readily done in the stable limit of large enough $s_j$.
\end{proof}

We do the following change of summation variables in formula \Ref{N decomp}.
Let $y=(y_1,\dots,y_{k-1})$ be a vector with coordinates:
\bea\label{y}
y_r=(ns_r-p_r)-(ns_{r+1}-p_{r+1})=\mdf{nl_r+a_r-b_r},
\qquad r=1,\dots,k-1.\hspace{-30pt}  
\ena
Let also
\be
\bar y=\sum_{r=1}^k(ns_r-p_r)=n(n-1)j+(n-1)\sum_{r=1}^k p_r+n\sum_{r=1}^ki_r,
\en
where 
\be
j=\frac{1}{n-1}\sum_{r=1}^k(s_r-i_r-p_r)\in\Z.
\en

Define
\be
w(j,i,p,\al,\beta)=\sum\limits_{r=1}^k\bigl(\eta^{(p_r)}+\eta_{i_r}\bigr)-j\eta,
\en
where we used the notation $i=(i_1,\dots,i_k)$.

Let $C_k$ be the Cartan matrix of $\mathfrak{sl}_k$. We have $(C_k^{-1})_{ir}=(C_k^{-1})_{ri}=i(k-r)/k$, where $1\leq i\leq r\leq k-1$.

\begin{cor}
We have an isomorphism of $\mathcal E_{n-1}^{\mdf{n-1|0}}\otimes \mathcal E_1^{\mdf{n-1||0}}$ modules
\bea\label{N decomp new}
\lefteqn{\mc N_{\al,\beta}^{(n;p)}(u)=\mathop{\oplus}\limits_{i_1,\dots,i_k=0}^{n-2}\mathop{\oplus}\limits_{j\in\Z}
 \left(
x^{\frac{\bar y^2}{2(n-1)k}+\frac{\bar y}{2}+\frac{1}{2(n-1)}\sum\limits_{r=1}^k(n-i_r-1)i_r}z^{w(j, i,p,\al,\beta)}\right)}\notag\\ &&
\times\left(\mathop{\oplus}\limits_{l_1,\dots,l_{k-1}=0}^\infty x^{\frac{y^tC_k^{-1} y }{2(n-1)}}\mc N_{\gamma( l),\beta}^{(n-1;n-1-i_k)}(-qq_1^{\frac{s}{n-1}}u)\boxtimes \mc N_{\gamma(l),\alpha}^{(1)}(-qq_3^{s}u)\right),
\ena
where 
\be
s=\frac{1}{k}(\bar y-\sum_{r=1}^{k-1}(r-k)y_r),
\en
while the summation is over $l_1,\dots, l_{k-1}$ such that
\be
\mdf{l_r+a_r}\equiv i_r-i_{r+1}+b_r \ (\on{mod} n-1)
\en
and
\be
\sum_{r=1}^{k-1} rl_r \equiv(n-1) j+\sum_{r=1}^k(i_r+p_r)+\sum_{r=1}^{k-1}\mdf{ rm_r}\ (\on{mod} k).
\en
\end{cor}
\begin{proof}
Formula \Ref{N decomp new} is obtained \Ref{N decomp} by the straightforward change of variables.
\end{proof}

\subsection{Macmahon modules}\label{M branch sec} In this section 
we discuss the  $k\to \infty$ limit of formula \Ref{N decomp new}.
 
 Fix partitions $\al,\beta$. Adding zero parts we can think that $\al,\beta$ have $k$ parts if $k$ is sufficiently large. Then, one can define the analytic continuation of the module $\mc N_{\al,\beta}^{(n;p)}(u)$ with respect to parameter $k$, the result is the so called Macmahon module
 $\mc M_{\al,\beta,\emptyset}^{(n;p)}(u,K)$. The Macmahon  module $\mc M_{\al,\beta,\emptyset}^{(n;p)}(u,K)$ is an admissible tame lowest weight $\mc E_n$-module of level $K$ which is irreducible for generic values of $K$, and whose basis is labeled by plane partitions with boundary conditions $\al,\beta,\emptyset$, see \cite{FJMM2}.

We conjecture the decomposition of $\mc M_{\al,\beta,\emptyset}^{(n;0)}(u,K)$ as $\mathcal E_{n-1}^{\mdf{n-1|0}}\otimes \mathcal E_1^{\mdf{n-1||0}}$ module based on \Ref{N decomp new} as follows.

Fix non-negative integers $l_1,\dots,l_t, l_1',\dots,l_t'$.
Let $L_k$ be the vector with $k$ components of the form $L=(l_1,\dots,l_t,0,\dots,0,l'_{t},l_{t-1}',\dots,l_1')$.

Similarly to the inductive construction of the Macmahon module, we expect the following.
\begin{conj}
There exists an $\mc E_n$ lowest weight admissible tame module
$\mdf{\mc M}_{\ga(l),\beta}^{\mdf{(n;p),\ga(l')}}(u,K)$ of level $K$ which is the analytic continuation of
$\mc N_{\ga(L_k),\beta}^{(p)}(u)$ with respect to $k$.
\end{conj}
 
Note that the module $\mdf{\mc M}_{\ga(l),\beta}^{\mdf{(n;p),\ga(l')}}(u,K)$ does not change if sequences
$l$ and $l'$ are extended by finitely many zeros. Namely, if $\tilde l=(l_1,\dots,l_t,0)$ and
$\tilde l'=(l_1',\dots,l_t',0)$ then 
$\mdf{\mc M}_{\ga(l),\beta}^{\mdf{(n;p),\ga(l')}}(u,K)
=\mdf{\mc M}_{\ga(\tilde l),\beta}^{\mdf{(n;p),\ga(l')}}(u,K)
=\mdf{\mc M}_{\ga(l),\beta}^{\mdf{(n;p),\ga(\tilde l')}}(u,K)$.
 
 \medskip
 \mdf{
If $l'=\emptyset$, this module is the Macmahon module:
$\mc M_{\ga(l),\beta}^{\mdf{(n;p),\ga(\emptyset)}}(u,K)=\mc M_{\ga(l),\beta}^{(n;p)}(u,K)$.
}
 
Recall that the parameters $a_i,b_i$ are given by \Ref{ab}, $p_i$ by \Ref{colors}, $m$ by  \Ref{m}, and $y$ by \Ref{y}.
If the partitions $\alpha,\beta$ have $t$ non-zero parts, we set $a_r=b_r=m_r=p_r=0$ for $r>t$. 

Let $G$ be the Gordon matrix given by $G_{i,j}=\min\{i,j\}$.

Then we have the following decomposition formula.

\begin{conj}
We have an isomorphism of $\mathcal E_{n-1}^{\mdf{n-1|0}}\otimes \mathcal E_1^{\mdf{n-1||0}}$ modules
\be
\lefteqn{\mc M_{\al,\beta,\emptyset}^{(n;0)}\mdf{(u,K)}
=\mathop{\oplus}\limits_{\buildrel{0\le i_r,i'_r\le n-2} 
\over {r=1,2,\cdots}}\mathop{\oplus}\limits_{j=-\infty}^{+\infty}
x^{\rho(i,i';j)}z^{w(i,i';j)}}
\\ 
&&
\times
\left(\mathop{\oplus}\limits_{\buildrel{l_r,l'_r,\dots\ge0} 
\over{r=1,2,\cdots}} x^{\frac{y^tG y +(y')^tG y'}{2(n-1)}}
\mdf{\mc M}_{\gamma( l),\beta}^{\mdf{(n-1;n-1-i_1'),\gamma(l')}}(-qq_1^{\frac{s}{n-1}}u\mdf{,K})
\boxtimes \mdf{\mc M}_{\gamma(l),\alpha}^{\mdf{(1),\gamma(l')}}(-qq_3^su\mdf{,K})\right),
\en
where $y'_r=nl_r'$, $s=\sum_{r\geq 1} y_r$,
\be
&\rho(i,i';j)=\frac{n(n-1)}{2}j
+\frac{(n-1)}{2}\sum\limits_{r\ge1} p_r
+\frac{n+1}{2}\sum\limits_{r\ge1} (i_r+i'_r)
-\frac{1}{2(n-1)}\sum\limits_{r\ge1} (i_r^2+(i'_r)^2),
\\
&
w(i,i';j)=\sum_{r\ge 1}\bigl(\eta^{(p_r)}+\eta_{i_r}+\eta_{i'_r}\bigr)-j\eta,
\en
and the summation is over $j\in\Z$, 
$i_r,i'_r\in\{0,\cdots,n-2\}$ and 
non-negative integers $l_r,l_r'$, $r\ge1$, 
such that only finitely many $i_r,i_r',l_r,l_r'$ are non-zero and 
\be
\mdf{l_r+a_r}\equiv i_r-i_{r+1}+b_r \ (\on{mod} n-1), \qquad l_r'=i_{r+1}'-i_r' \ (\on{mod} n-1),
\en
and
\be
\sum_{r\geq 1} rl_r -\sum_{r\geq 1} rl_r'=(n-1) j+\sum_{r\geq 1}(i_r+p_r+i_r')+\sum_{r\geq 1}
\mdf{rm_r}.
\en
\end{conj}

\bigskip

 {\bf Acknowledgments.}
The authors thank Igor Burban for drawing their attention to the paper
\cite{BS} where the quantum toroidal $\mathfrak{gl}_1$ algebra has been 
introduced and its basic properties are proved.

Research of MJ is supported by the Grant-in-Aid for Scientific Research B-23340039.
Research of TM is supported by the Grant-in-Aid for Scientific Research B-22340031.

The present work has been carried out during the visits of BF and EM 
to Kyoto University. They wish to thank the University for hospitality.

\end{document}